\documentclass[a4paper, 12pt]{amsart}

\usepackage{amsmath,amssymb,amsthm,url,scalerel}
\usepackage[utf8]{inputenc}
\usepackage[T1]{fontenc}
\usepackage{libertine}
\usepackage[libertine,cmintegrals,cmbraces]{newtxmath}
\usepackage{verbatim}
\usepackage[shortlabels]{enumitem}
\usepackage{stmaryrd}
\usepackage{mathtools}
\usepackage{microtype}
\usepackage{tabto}
\usepackage{xcolor}
\usepackage{tikz-cd}
\usepackage{mathrsfs}
\usepackage{nicefrac}
\usepackage{tabularx}
\usepackage{dsfont}
\usepackage{todonotes}
\usepackage{a4wide}
\usepackage{euscript}
\usepackage[all]{xy}
\usepackage{mathrsfs,comment}
\numberwithin{equation}{subsection}
\usepackage{multirow}

\AtBeginDocument{%
   \def\MR#1{}
}

\usepackage[pagebackref]{hyperref}
  
\hypersetup{%
  bookmarksnumbered=true,
  colorlinks=true,%
  linkcolor=blue,%
  citecolor=blue,%
  filecolor=blue,%
  menucolor=blue,%
  urlcolor=blue,%
  pdfnewwindow=true,%
  pdfstartview=FitBH}

\definecolor{USCG}{HTML}{FFCC00}
\definecolor{USCR}{HTML}{990000}

\usepackage{xcolor}
\definecolor{seagreen}{RGB}{46,139,87}
\definecolor{maroon}{RGB}{128,0,0}
\definecolor{darkviolet}{RGB}{148,0,211}
\definecolor{twelve}{RGB}{100,100,170}
\definecolor{thirteen}{RGB}{100,150,50}
\definecolor{fourteen}{RGB}{200,0,0}
\definecolor{fifteen}{RGB}{0,200,0}
\definecolor{sixteen}{RGB}{0,0,200}
\definecolor{seventeen}{RGB}{200,0,200}
\definecolor{eighteen}{RGB}{0,200,200}

\newcommand{\bb}[1]{\mathbb{#1}}

\newcommand{\mbf}[1]{\mathbf{#1}}

\newcommand{\ms}[1]{\mathscr{#1}}
\newcommand{\es}[1]{\EuScript{#1}}
\renewcommand{\sf}[1]{\mathsf{#1}}

\DeclareMathOperator{\colim}{\mathrm{colim}}

\DeclareMathOperator{\hocolim}{\mathrm{hocolim}}
\DeclareMathOperator{\holim}{\mathrm{holim}}
\DeclareMathOperator{\fiber}{\mathrm{fib}}

\newcommand{\fib}{\mathsf{fib}}

\newcommand{\s}{{\sf{Sp}}}
\DeclareMathOperator{\T}{\es{S}}
\newcommand{\poly}[1]{\mathsf{Poly}^{\leq #1}}
\newcommand{\homog}[1]{\mathsf{Homog}^{#1}}

\newcommand{\Aut}{\mathrm{Aut}}
\newcommand{\Fun}{\sf{Fun}}
\DeclareMathOperator{\Hom}{\mathsf{Hom}}

\DeclareMathOperator{\Map}{\mathsf{Map}}
\DeclareMathOperator{\mor}{\mathsf{mor}}
\DeclareMathOperator{\nat}{\mathsf{nat}}

\DeclareMathOperator{\id}{\mathrm{Id}}

\DeclareMathOperator{\sing}{\mathrm{Sing}}

\newcommand{\R}{\mbf{R}}

\newcommand{\op}{\mathrm{op}}

  \newcommand{\adjunction}[4]{
\xymatrix{
#1:#2 \ar@<.5ex>[r] &
\ar@<.5ex>[l] #3:#4
}}

\newtheorem{thm}{Theorem}[subsection]
\newtheorem{prop}[thm]{Proposition}
\newtheorem{lem}[thm]{Lemma}
\newtheorem{cor}[thm]{Corollary}

\newtheorem{hypothesis}[thm]{Hypothesis}

\theoremstyle{definition}
\newtheorem{definition}[thm]{Definition}
\newtheorem{ex}[thm]{Example}
\newtheorem{exs}[thm]{Examples}
\newtheorem{rem}[thm]{Remark}

\newtheorem{xxthm}{Theorem}

\begin{document}


\title{Symplectic Weiss calculi}

\author{Matthew Carr}
\address[Carr]{Department of Population and Public Health Sciences, University of Southern California}
\email{mbcarr@usc.edu}

\author{Niall Taggart}
\address[Taggart]{Radboud University Nijmegen}
\email{niall.taggart@ru.nl}

\date{\today}



\begin{abstract}
We provide two candidates for symplectic Weiss calculus based on two different, but closely related, collections of groups. In the case of the non-compact symplectic groups, i.e., automorphism groups of vector spaces with symplectic forms, we show that the calculus deformation retracts onto unitary calculus as a corollary of the fact that Weiss calculus only depends on the homotopy type of the groupoid core of the diagram category. In the case of the compact symplectic groups, i.e., automorphism groups of quaternion vector spaces, we provide a comparison with the other known versions of Weiss calculus analogous to the comparisons of calculi of the second author, and classify certain stably trivial quaternion vector bundles over finite cell complexes in a range, using elementary results on convergence of Weiss calculi.
\end{abstract}
\maketitle

\setcounter{tocdepth}{1}
{\hypersetup{linkcolor=black} \tableofcontents}

\section{Introduction}

Weiss calculi are a family of homotopy theoretic tools designed to study geometric problems arising from differential topology. The original example of such a calculus was \emph{orthogonal calculus}~\cite{Weiss}, which provides a calculus for functors from real inner product spaces to spaces. Other variants now include \emph{unitary calculus}~\cite{TaggartUnitary} which studies the corresponding complex geometry, \emph{unitary calculus with Reality}~\cite{TaggartReality} which takes into account the symmetry in complex geometry given by complex conjugation, general \emph{equivariant Weiss calculi}~\cite{Yavuz,BhattacharyaHu} and \emph{$FI$-calculus}~\cite{Arro} which considers functors from finite sets to any stable $\infty$-category, and is closely related to representation stability. 

Applications of various versions of Weiss calculus abound geometry and homotopy theory. For instance, on the geometric side, Krannich and Randal-Williams~\cite{KrannichRandal-Williams} compute the (rational) second orthogonal derivative of $\mathsf{BTop}(\R^d)$, the classifying space of topological rank $d$-bundles, and use this to compute the rational homotopy type of the group of diffeomorphisms of a closed $d$-dimensional disc fixing the boundary in a specified range, and determine the optimal rational concordance stable range for high-dimensional discs. On the more homotopy-theoretic side, versions of Weiss calculus open up the possibility of encoding certain levels of functoriality, which are not present on the level of spaces. For instance, this kind of functoriality has been utlised by Arone~\cite{AroneAkfree} in their construction of finite type $n$ spectra, by Behrens~\cite{BehrensEHP} in relating the EHP sequence to Goodwillie calculus, and by Kuhn~\cite{KuhnWhitehead} when studying the Whitehead Conjecture. 

With applications like these in mind, the question of when a category of functors admit a version of Weiss calculus is an interesting one, which has received much attention. For example, Anel, Biedermann, Finster and Joyal~\cite{ABFJ} claim to recover orthogonal calculus from their topos-theoretic approach to generalised Goodwillie calculus. 

In this article, we consider two natural extensions of Weiss calculus to symplectic geometry. The first is to consider the potential of a version of Weiss calculus for functors from vector spaces with a symplectic form to spaces, i.e., a version of Weiss calculus based on \emph{non-compact symplectic groups}. We call this version \emph{symplectic Weiss calculus}. The second  is to consider a version of Weiss calculus built on the quaternions rather than Euclidean spaces, i.e., a version of Weiss calculus based on \emph{compact symplectic groups}. We call this version \emph{quaternion Weiss calculus}. 

\subsection*{Symplectic Weiss calculus}
In this version of Weiss calculus we consider functors from the category of vector spaces with symplectic forms to spaces. Part~\ref{part: non-compact} can be summarised in the following result. 

\begin{xxthm}
    A version of Weiss calculus exists for functors from vector spaces with a symplectic form to spaces, and this Weiss calculus is equivalent to unitary calculus.
\end{xxthm}

The intricacies of linear algebra in the absence of an inner product lead one to require a new method for constructing the symplectic calculus. Taking motivation from homological stability, see for example,~\cite{RWW}, we show that Weiss calculus only depends on the homotopy type of the groupoid core of the `indexing category' by exploiting Quillen's bracket construction\footnote{Quillen's bracket construction produces from a groupoid $S$ a small symmetric monoidal category $\langle S, S \rangle$ by considering the action of the groupoid on itself. By considering the groupoid $S$ acting on $S \times S$ one recovers what is classically known as Quillen's $S^{-1}S$-construction, which in our notation would be the category $\langle S, S\times S\rangle$. We will make no use of the latter construction in this paper.}. The following table describes the relationship between known versions of Weiss calculus and the groupoid cores of the indexing categories. 

\begin{center}
    \begin{tabular}{|p{4cm}|p{5cm}|p{6cm}|}
    \hline
    \textbf{Weiss calculus} & \textbf{Groupoid core} & \textbf{Quillen's bracket construction} \\
    \hline
    Orthogonal~\cite{Weiss} & Finite-dimensional real inner product spaces with linear isometric isomorphisms & Finite-dimensional real inner product spaces with linear isometric embeddings \\
    \hline 
    Unitary~\cite{TaggartUnitary} & Finite-dimensional Hermitian inner product spaces with linear isometric isomorphisms & Finite-dimensional Hermitian inner product spaces with linear isometric embeddings \\
    \hline 
    $\sf{FI}$~\cite{Arro} & Finite sets and bijections & Finite sets and injections \\
    \hline 
    Symplectic (Part~\ref{part: non-compact}) & Finite-dimensional symplectic vector spaces with symplectic isomorphisms & Finite-dimensional symplectic vector spaces with symplectic linear maps \\
    \hline
    Quaternion (Part~\ref{part: quaternion calc}) & Finite-dimensional quaternion inner product spaces with linear isometric isomorphisms & Finite-dimensional quaternion inner product spaces with linear isometric embeddings \\
    \hline 
    \end{tabular}
\end{center}

Denote by $\sf{G}^\mbf{S}$ the groupoid of symplectic vector spaces and symplectic isomorphisms and denote by $\es{J}^\mbf{S}$ Qullien's bracket construction on $\sf{G}^\mbf{S}$, as above. Similarly denote by $\sf{G}^\mbf{U}$ the groupoid of Hermitian inner product spaces and linear isometric isomorphisms, and by $\es{J}^\mbf{U}$ the corresponding bracket construction.  There is a functor $\Im: \sf{G}^\mbf{U} \to \sf{G}^\mbf{S}$ which sends a Hermitian inner product space to its underlying real vector space with symplectic structure given by the imaginary part of the Hermitian inner product. On morphism spaces this functor is little more than the inclusion of the unitary group $U(n)$ into the non-compact symplectic group $\mathsf{Sp}(2n,\R)$. A routine linear algebra exercise exhibits $U(n)$ as the maximal compact subgroup of $\mathsf{Sp}(2n, \R)$, and hence the inclusion $U(n) \subseteq \mathsf{Sp}(2n,\R)$ is a homotopy equivalence, or said differently, the non-compact symplectic group deformation retracts onto the unitary group. From this, it follows that the functor $\Im: \sf{G}^\mbf{U} \to \sf{G}^\mbf{S}$ is an equivalence of $\infty$-groupoids, and hence induces an equivalence $\Im : \es{J}^\mbf{U} \to \es{J}^\mbf{S}$ of $\infty$-categories by Lemma~\ref{lem: equiv detected on groupoid cores}. This equivalence allows us to sidestep any linear algebra in the construction of symplectic calculus, producing a Weiss tower
\[\begin{tikzcd}
	&& F \\
	\cdots & {T_nF} & \cdots & {T_1F} & {T_0F}
	\arrow[from=2-4, to=2-5]
	\arrow[from=2-2, to=2-3]
	\arrow[from=2-3, to=2-4]
	\arrow[bend right=30, from=1-3, to=2-2]
	\arrow[from=2-1, to=2-2]
	\arrow[bend left=30, from=1-3, to=2-4]
	\arrow[bend left=20, from=1-3, to=2-5]
\end{tikzcd}\]
for any symplectic functor $F: \es{J}^\mbf{S} \to \T$. For any choice of category of vector spaces, the Weiss tower may be packaged as a functor
\[
\mathsf{Tow} : \Fun(\es{J}, \T) \longrightarrow \Fun(\mbf{Z}_{\geq 0}^\op,\Fun(\es{J}, \T)),
\]
and the equivalence between the symplectic and unitary indexing categories implies that the symplectic Weiss tower may be identified with the unitary Weiss tower in the following sense. 

\begin{xxthm}
For any symplectic functor $F$, there is an equivalence
    \begin{align*}
    \mathsf{Tow}^\mbf{S}(F) &\simeq \Im_! \mathsf{Tow}^\mbf{U}(\Im^\ast F),
    \end{align*}
where $\Im_!$ is the $(\infty$-categorical) inverse of pullback along $\Im: \es{J}^\mbf{U} \to \es{J}^\mbf{S}$.
\end{xxthm}

\subsection*{Quaternion Weiss calculus}
For the version of Weiss calculus built from linear algebra over the quaternions, the key insight is that although the quaternions are not a field, quaternionic vector spaces support just enough linear algebra to
make all the standard constructions of Weiss~\cite{Weiss} go through. We obtain the following summary of the results in Part~\ref{part: quaternion calc}.

\begin{xxthm}
    A version of Weiss calculus exists for functors from quaternion inner product spaces to (pointed) spaces.
\end{xxthm}

The layers of the Weiss tower are \emph{homogeneous of degree $n$}, and we can classify such functors in terms of spectra with an action of $\mathsf{Sp}(n)=\Aut(\mbf{H}^n)$ completely analogously to how homogeneous functors of degree $n$ in orthogonal calculus are classified by spectra with an action of $O(n)$.

Through this classification we can compare quaternion calculus with both unitary and orthogonal calculus. In particular, we can construct a functor from unitary calculus to compact symplectic calculus which preserves Weiss towers: there is a functor $h: \es{J}^\mbf{H} \to \es{J}^\mbf{U}$, right adjoint to the extension-of-scalars functor $\mbf{H} \otimes_\mbf{C}(-)$, which by precomposition induces a functor 
\[
h^\ast: \Fun(\es{J}^\mbf{U}, \T_\ast) \longrightarrow \Fun(\es{J}^\mbf{H}, \T_\ast),
\]
through which we compare quaternion calculus and unitary calculus. 

\begin{xxthm}\label{thm: calc def retracts}
Let $F$ be a unitary functor. There is a levelwise equivalence
    \[
    h^\ast \mathsf{Tow}^\mbf{U}(F) \simeq \mathsf{Tow}^\mbf{H}(h^\ast F). 
    \]
\end{xxthm}

\subsection*{Classification of quaternion vector bundles}
As an application of the theory of quaternion calculus we classify certain quaternion vector bundles over finite-dimensional cell complexes in a specified range. This extends work of Hu~\cite{Hu} from complex bundles to quaternion bundles, although or methods and ranges vary slightly. Our methods use concrete analyticity bounds for the functor $\mathsf{BSp}(-)$ which sends a quaternion vector space to the classifying space of the compact symplectic group $\mathsf{Sp}(V)=\mathsf{Aut}(V)$, and the observation that rank $r$ quaternion bundles over a finite-dimensional cell complex $X$ are classified by $[X, \mathsf{BSp}(r)] = \pi_0\Map_\ast(X, \mathsf{BSp}(r))$. Denote by $\mathsf{KSp}$ quaternionic topological $K$-theory, i.e., the version of topological $K$-theory represented by $\mathsf{BSp}$.

\begin{xxthm}\label{thm: E}
Let $X$ be a $d$-dimensional cell complex for which $\widetilde{\mathsf{KSp}}^{-1}(X)=0$. The set of stably trivial rank $r$ quaternion vector bundles over $X$ where $\frac{d+1}{8}\leq r < \frac{d-2}{4}$ is given by $\{X, \Sigma^3\mbf{H}P^\infty_r\}$, the set of stable maps from $X$ to the three-fold suspension of stunted quaternion projective space.
\end{xxthm}

\subsection*{The use of $\infty$-categories}
In this paper we will work with multiple models for $\infty$-categories and functors between them. Typically by an $\infty$-category we will mean a quasicategory, and we aim to make clear when we are using another model (for example, complete Segal spaces). For brevity, we will often refer to homotopy (co)limits in any chosen model for $\infty$-categories simply as (co)limits, unless confusion is likely to occur. In particular all (co)fiber sequences are homotopy (co)fiber sequences.

Many of our $\infty$-categories will come to us from model categories: let $\es{M}$ be a simplicial model category. The underlying $\infty$-category $\es{M}_\infty=N^h(\es{M}^\circ)$ presented by $\es{M}$ is obtained by taking the homotopy coherent nerve of its full simplicial subcategory of bifibrant objects. The same is true at the level of adjunctions: a simplicial Quillen adjunction 
\[
\adjunction{F}{\es{M}}{\es{N}}{G},
\]
induces an adjunction of $\infty$-categories
\[
\adjunction{F}{\es{M}_\infty}{\es{N}_\infty}{G},
\]
which is an equivalence if the Quillen adjunction is a Quillen equivalence. For more details and for slight generalizations see for example,~\cite{Mazel-GeeAdjunctions}. Note that one can replace ``simplicial'' with ``topological'' above by virtue of~\cite[Corollary 2.7]{Ilias}, and an analogous procedure will send a topological model category to an $\infty$-category. Unless confusion is likely to occur, we will drop the subscript $\infty$ from our notation.

We will denote by $\es{S}$ the $\infty$-category of spaces, and we will denote by $\sf{Top}$ the category of $\Delta$-generated spaces\footnote{The category of $\Delta$-generated spaces is a convenient model for spaces, for details see for example,~\cite{DuggerDeltaGeneratedSpaces}.} equipped with the Quillen model structure, so that one model for $\es{S}$ is $\sf{Top}_\infty$. Given a small topological category $\es{J}$, we denote by $\Fun(\es{J}, \es{S})$ the $\infty$-category of functors from $\es{J}$ to spaces. There are many models for this $\infty$-category, including the underlying $\infty$-category of the projective model structure on the category of topologically enriched functors from $\es{J}$ to $\sf{Top}$. This model categorical presentation is (simplicially) Quillen equivalent to the projective model structure on the category of simplicially enriched functors from $\sing(\es{J})$ to simplicial sets, where $\sing(\es{J})$ is the simplicially enriched category with objects the objects of $\es{J}$ and morphism simplicial set given by applying the singular simplicial set functor to the morphism topological spaces of $\es{J}$. In~\cite[\S5.1.1]{HTT}, Lurie provides other models for this $\infty$-category, including as the mapping space $\Fun(N^h(\es{J}), \T)$ from the homotopy coherent nerve of $\es{J}$ to the $\infty$-category $\T$, and the homotopy coherent nerve of the category of left fibrations over $N^h(\es{J})$. 

We will often need to talk about homotopy (co)limits of internal presheaves, i.e., functors from a topologically (resp. simplicially) internal category to the category of topological spaces (resp. simplicial sets). This is where complete Segal spaces, either topological or simplicial, enter the picture as the topologically (resp. simplicially) internal nerve of a topologically (resp. simplicially) internal category is a (non-Reedy fibrant) complete Segal space. The complete Segal model structures on simplicial topological spaces and bisimplicial sets are (simplicially) Quillen equivalent and hence we will drop the distinction between topological and simplicial, for more on this see~\cite{CarrTaggartcolimits}.

We will model functors out of a Segal space $B$ by using the left fibration model structure on the category of Segal spaces over $B$. In the simplicial setting this is worked out in detail by Boavida de Brito~\cite{BdB}, and we provide the analogous considerations in the topological settings in~\cite{CarrTaggartcolimits}. Note that if $B$ is non-Reedy fibrant, then the left fibration model structure over $B$ is Quillen equivalent to the left fibration model structure over a Reedy fibrant replacement of $B$. By~\cite[Theorem 1.22]{BdB} there is an equivalence between left fibrations of complete Segal spaces over $B$ and left fibrations of quasicategories over the underlying quasicategory of $B$. For geometric applications like ours it is often more convenient to work on the complete Segal side of this equivalence.

\subsection*{Acknowledgements}

This work had greatly benefited from numerous conversations with Kaya Arro, and from helpful discussions with Oscar Randal-Williams, Steffen Sagave and Michael Weiss. We are also grateful to Maxwell Johnson for pointing out an error in the statement of Theorem~\ref{thm: E}. MC gratefully acknowledges partial support from NSF-DMS \#1547357. NT was supported by the European Research council (ERC) through the grant “Chromatic homotopy theory of spaces”, grant no. 950048 and by the Nederlandse Organisatie voor Wetenschappelijk Onderzoek (Dutch Research Council) Vidi grant no VI.Vidi.203.004 in the final stages of this project. This material is based upon work supported by the Swedish Research Council under grant no.2016-06596 while NT was in residence at Institut Mittag-Leffler in Djursholm, Sweden as part of the program “Higher algebraic structures in algebra, topology and geometry” in 2022. NT would like to thank the Isaac Newton Institute for Mathematical Sciences, Cambridge, for support and hospitality during the programme \emph{Equivariant homotopy theory in context}, where work on this paper was undertaken.

\part{Symplectic Weiss calculus}\label{part: non-compact} 
In the first part of this paper we concentrate on the symplectic Weiss calculus, i.e., Weiss calculus built on the category of symplectic vector spaces  and symplectic linear maps. We show through a general discussion on the relationship between Weiss calculus and Quillen's bracket construction, that symplectic calculus deformation retracts onto unitary calculus. 

\section{Weiss calculus and groupoid cores}
\label{sec: groupoid cores}

Given a monoidal groupoid, Quillen's bracket construction provides a small category, for details see for example,~\cite{Grayson} or~\cite[\S1]{RWW}. In this section we observe that the indexing categories (i.e., the categories of vector spaces) which govern the many variants of Weiss calculus are naturally of this form, a relationship that the second author first learned from O. Randal-Williams. From this observation we show that Weiss calculus is completely determined by the homotopy type of the groupoid core of the indexing category.

\subsection{Quillen's bracket construction}
Let $(\sf{G}, \oplus, 0)$ be a topologically enriched monoidal groupoid, i.e., the tensor product lifts to map
\[
\oplus: \Hom_\sf{G}(A,B) \times \Hom_\sf{G}(C,D) \to \Hom_\sf{G}(A \oplus C, B \oplus D),
\]
and, in particular, it is an enriched functor. Define $\langle \sf{G}, \sf{G} \rangle$ to be the category with the same objects as $\sf{G}$ and a morphism $[X,f]: A \to B$ in $\langle \sf{G}, \sf{G} \rangle$ is an equivalence class of a pair $(X,f)$ where $X \in \sf{G}$ and $f: X \oplus A \to B$, under the equivalence relation $(X,f) \sim (X',f')$ if and only if there exists and isomorphism $g: X \to X'$ in $\sf{G}$ such that the diagram
\[\begin{tikzcd}
X \oplus A \arrow[r, "f"] \arrow[d, "g \oplus A"'] & B \\
X' \oplus A \arrow[ur, "f'"']
\end{tikzcd}\]
commutes. The composite
\[
A \xrightarrow{[X,f]} B \xrightarrow{[Y, g]} C
\]
is the map
\[
A \xrightarrow{[Y \oplus X, ~g \circ (Y \oplus f)]} C.
\]
The category $\langle \sf{G}, \sf{G} \rangle$ is  enriched in spaces having mapping spaces defined by
\[
\Hom_{\langle \sf{G}, \sf{G} \rangle}(A,B) = \underset{X \in \sf{G}}{\colim}~\Hom_\sf{G}(X \oplus A, B). 
\]

\begin{rem}
For our examples of interest, the groupoid $\sf{G}$ with always be symmetric monoidal. In particular, in our examples, we will not have to take care over which side we let the groupoid act on itself in order to form the bracket category. In~\cite{RWW}, Randal-Williams and Wahl make no symmetric assumption and always have the groupoid acting on the left, or impose what they call \emph{pre-braiding}.
\end{rem}

\begin{exs}\label{exs: groupoid cores and S-1S}\hspace{10cm}
\begin{enumerate}
    \item If $\sf{G} = \Sigma$ the groupoid of finite sets and bijections (with the discrete topology), then $\langle \sf{G}, \sf{G} \rangle \cong \sf{FI}$ is the category of finite sets and injections. A version of orthogonal calculus based on $\sf{FI}$ was studied by Arro~\cite{Arro}.
    \item If $\sf{G} = \sf{G}^\mbf{O}$ the groupoid of finite-dimensional real inner product spaces and linear isometric isomorphisms, then $\langle \sf{G}, \sf{G} \rangle \cong \es{J}^\mbf{O}$ is the category of finite-dimensional real inner product spaces and linear isometric embeddings. Calculus built on $\es{J}^\mbf{O}$ is the orthogonal calculus of Weiss~\cite{Weiss}. 
   \item  If $\sf{G} = \sf{G}^\mbf{U}$ the groupoid of finite-dimensional complex inner product spaces and linear isometric isomorphisms, then $\langle \sf{G}, \sf{G} \rangle \cong \es{J}^\mbf{U}$ is the category of finite-dimensional complex inner product spaces and linear isometric embeddings. Unitary calculus in the sense of the second author~\cite{TaggartUnitary} is the Weiss calculus based on $\es{J}^\mbf{U}$.
  \item  If $\sf{G} = \sf{G}^\mbf{S}$ the groupoid of finite-dimensional symplectic vector spaces and symplectic isomorphisms, then $\langle \sf{G}, \sf{G} \rangle \cong \es{J}^\mbf{S}$ is the category of finite-dimensional symplectic vector spaces and symplectic linear maps. In this first part of the paper, we will construct a Weiss calculus on this category 
  \item if $\sf{G} = \sf{G}^\mbf{H}$ is the groupoid of finite-dimensional quaternion inner product spaces and isometric isomorphisms, then $\langle \sf{G}, \sf{G} \rangle \cong \es{J}^\mbf{H}$ is the category of finite-dimensional quaternion inner product spaces with linear isometric embeddings. In Part~\ref{part: quaternion calc} we will construct a version of Weiss calculus based on $\es{J}^\mbf{H}$. 
    \end{enumerate}
\end{exs}
\begin{proof}
In all cases, the proof is analogous to that provided by Randal-Williams and Wahl~\cite[\S5.1]{RWW} in the case of symmetric groups.
\end{proof}

\subsection{Groupoid cores}
Under some mild hypotheses the monoidal groupoid $\sf{G}$ is the ($\infty$-)groupoid core of the category $\langle \sf{G}, \sf{G} \rangle$. We first recall the notion of a Dwyer-Kan equivalence of topologically enriched categories.

\begin{definition}
    A functor $F: \es{C} \to \es{D}$ of topologically enriched categories is a \emph{Dwyer-Kan} equivalence if the following two conditions hold.
    \begin{enumerate}
        \item $F$ is \textit{fully faithful}\textemdash that is, the morphism
        \[
        \es{C}(X,Y) \longrightarrow \es{D}(F(X), F(Y)),
        \]
        is a weak homotopy equivalence;  
        \item $F$ is \textit{homotopically essentially surjective}\textemdash that is, the induced functor
        \[
        \pi_0F : \pi_0\es{C} \longrightarrow \pi_0\es{D},
        \]
        on homotopy categories is essentially surjective.
    \end{enumerate}
\end{definition}

To identify the ($\infty$-)groupoid core with the groupoid $\sf{G}$ we will say a monoidal category $(\mathsf{C},\oplus,0)$ is \emph{cancellative} if whenever $X\oplus A\cong Y\oplus A$, then $X\cong Y$, and \emph{has no zero divisors} if whenever $C\oplus C'\cong 0$, then $C \cong 0\cong C'$. These are properties that will be satisfied  by all of the groupoids we consider. 

\begin{lem}\label{lem: groupoid is core}
Let $G$ be a topologically enriched monoidal groupoid. There is an enriched functor 
\[
\mathsf{I}: \sf{G} \longrightarrow \langle \sf{G}, \sf{G} \rangle^\mathsf{core}, \ (f: A \to B) \longmapsto [0, f],
\]
which 
\begin{enumerate}
    \item is faithful if $\Aut_\sf{G}(0) = \{\id\}$,
    \item is full if $\sf{G}$ has no zero divisors,
    \item exhibits $\sf{G}$ as the groupoid core of $\langle \sf{G}, \sf{G} \rangle$\textemdash that is, is an equivalence\textemdash whenever $\Aut_\sf{G}(0) = \{\id\}$ and $\sf{G}$ has no zero divisors,
    \item induces an equivalence on homotopy categories whenever $\sf{G}$ is cancellative with no zero divisors and $\Aut_\sf{G}(0)=\id$. In particular, the isomorphisms in $\pi_0\langle \sf{G}, \sf{G} \rangle$ are isomorphisms in $\langle \sf{G}, \sf{G} \rangle$ and hence $(\pi_0\langle\sf{G},\sf{G}\rangle)^{\sf{core}}\cong \pi_0(\langle\sf{G},\sf{G}\rangle^{\sf{core}})$, and
    \item is a Dwyer-Kan equivalence whenever $\sf{G}$ is cancellative with no zero divisors and $\Aut_\sf{G}(0)=\id$.
\end{enumerate}
\end{lem}
\begin{proof}
Continuity of $\sf{I}$ follows as it is given on arrows as the composite 
\[
\Hom_{\sf{G}}(A,B)\cong \Hom_{\sf{G}}(0\oplus A,B)\to \Hom_{\langle\sf{G},\sf{G}\rangle}(A,B),
\]
using the colimit structure map and the unit isomorphism. Statements $(1)$\textendash$(3)$ follow from \cite[Proposition 1.7]{RWW}. For $(4)$, we first show the isomorphisms in the homotopy category are represented by isomorphisms of $\langle\sf{G},\sf{G}\rangle$. By $(3)$, all such isomorphisms take the form $[0,f]$ under the equivalence $\sf{I}$.

If $[X,f]\colon A\to B$ is an isomorphism in $\pi_0\langle\sf{G},\sf{G}\rangle$, then there exists $[Y,g]\colon B\to A$ for which there are paths $[Y\oplus X,g\circ (Y\oplus f)]\simeq [0,\id_A]$ and $[X\oplus Y,f\circ (X\oplus g)]\simeq [0,\id_B]$ in $\Hom_{\langle\sf{G},\sf{G}\rangle}(A,A)$ and $\Hom_{\langle\sf{G},\sf{G}\rangle}(B,B)$, respectively. Since both composites $[X,f]\circ[Y,g]$ and $[Y,g]\circ[X,f]$ are self-maps in $\langle\sf{G},\sf{G}\rangle$, the cancellative property implies $X\oplus Y\cong Y\oplus X\cong 0$ and since there are no zero divisors, we must have $X\cong Y\cong 0$. But this means that $[X,f]$ and $[Y,g]$ are equivalent to maps of the form $[0,F]$ and $[0,G]$, respectively, as desired. Hence, all isomorphisms in $\pi_0\langle \sf{G},\sf{G}\rangle$  are represented by isomorphisms of $\langle \sf{G},\sf{G}\rangle$ and thus $\sf{G}$. In particular, this implies that $(\pi_0\langle\sf{G},\sf{G}\rangle)^{\sf{core}}\cong \pi_0(\langle\sf{G},\sf{G}\rangle^{\sf{core}})$. $(5)$ follows from $(4)$ and the additional observation that, under our hypotheses, $\sf{I}$ also induces equivalences between mapping spaces.
\end{proof}

The modest assertions of the preceding lemma belie its more far-reaching implications, as suggested by following illustrative examples.

\begin{ex}
In the examples of Example~\ref{exs: groupoid cores and S-1S} the groupoid $\sf{G}$ is the groupoid core of $\langle \sf{G}, \sf{G}\rangle$.
\end{ex}

\begin{rem}
When $\sf{G}$ is a topologically enriched cancellative monoidal groupoid, Lemma \ref{lem: groupoid is core}(4) implies that a topologically enriched functor $F\colon \mathcal{C}\to \langle\sf{G},\mathsf{G}\rangle$ is homotopically essentially surjective if and only if $F$ essentially surjective. 
\end{rem}

\subsection{Functors on groupoid cores}
In the original incarnation of Weiss calculus, one studies functors from $\es{J}^\mathbf{O}$ to spaces, which by Example~\ref{exs: groupoid cores and S-1S} is precisely studying functors from $\langle \mathsf{G}^\mathbf{O}, \mathsf{G}^\mathbf{O}\rangle$ to spaces. Our goal of rephrasing Weiss calculus in terms of these homogeneous categories is to invert a functor of topological groupoids and hence a functor on the associated homogeneous categories.

\begin{lem}\label{lem: equiv of diagrams induces equiv of functors}
Let $\es{J}_1$ and $\es{J}_2$ be small categories. If there is a Dwyer-Kan equivalence $\alpha : \es{J}_1 \to \es{J}_2$, then there is an adjunction
\[
\adjunction{\alpha_!}{\Fun(\es{J}_1, \T)}{\Fun(\es{J}_2, \T)}{\alpha^\ast},
\]
which is an equivalence of $\infty$-categories.
\end{lem}
\begin{proof}
The Dwyer-Kan equivalence $\alpha: \es{J}_1 \to \es{J}_2$ induces a Dwyer-Kan equivalence of simplicially enriched categories $\sing(\alpha): \sing(\es{J}_1) \to \sing(\es{J}_2)$, where for $i \in \{1,2\}$, $\sing(\es{J}_i)$ is the simplicially enriched category formed by applying $\sing$ to the morphism spaces of $\es{J}_i$. This in turn induced a categorical equivalence 
\[
N^h(\alpha): N^h(\es{J}_1) \longrightarrow N^h(\es{J}_2),
\]
on homotopy coherent nerves (where we have suppressed $\sing$ from the notation) and the result follows from~\cite[Proposition 1.2.7.3(3)]{HTT}, see also~\cite[\href{https://kerodon.net/tag/01E7}{Tag 01E7}]{kerodon}.
\end{proof}


To reduce to groupoid cores, we will need to impose some modest hypotheses\textemdash which hold in all of the cases we will be interested in\textemdash to ensure good homotopical behaviour. Note that in conjunction with $\sf{G}$ having no zero divisors, the following hypothesis makes $\langle \sf{G} , \sf{G} \rangle$ into a homogeneous category in the sense of Randal-Williams and Wahl~\cite[Definition 1.3]{RWW}, see also~\cite[Theorem 1.10]{RWW}.

\begin{hypothesis}\label{hyp: groupoid}
Let $\sf{G}$ be a topologically enriched monoidal groupoid which is cancellative and be such that $\Aut(A) \to \Aut(A \oplus B)$ is injective for all objects $A, B \in \sf{G}$.
\end{hypothesis}

Under these assumptions, we have the following lemma.

\begin{lem}\label{lem: equiv detected on groupoid cores}
Let $\sf{G}_1$ and $\sf{G}_2$ satisfy Hypothesis~\ref{hyp: groupoid}. If there is a monoidal Dwyer-Kan equivalence $\alpha \colon \sf{G}_1 \to \sf{G}_2$, then the induced functor
\[
\langle \alpha \rangle: \langle \sf{G}_1, \sf{G}_1 \rangle \longrightarrow \langle \sf{G}_2, \sf{G}_2 \rangle,
\]
is a Dwyer-Kan equivalence.
\end{lem}
\begin{proof}
The functor $\langle\alpha\rangle$ is defined on objects as $\alpha$ is and inherits its enrichment from $\alpha$.
Note that the colimit defining $\Hom_{\langle\sf{G}_i,\sf{G}_i\rangle}(A_i,B_i)$ is equivalent, by the cancellative property, to the colimit taken over the full subgroupoid of $\sf{G}_i$ spanned by objects having the isomorphism type of $X_i\in\sf{G}_i$ where $X_i$ is such that $X_i\oplus A_i\cong B_i$\textemdash all other mapping spaces are necessarily empty and do not affect the colimit. Call this subcategory $\sf{G}_i(X)$ and let $B\Aut(X)$ denote the full subcategory of $\sf{G}_i(X)$ spanned by $X$. Since the inclusion $B\Aut(X)\to \sf{G}_i({X})$ is an equivalence, it is final, and we conclude that 
\[
\Hom_{\langle\sf{G}_i,\sf{G}_i\rangle}(A_i,B_i)\cong \mathop{\colim}\limits_{B\Aut(X_i)}\Hom_{\sf{G}_i}(X_i\oplus A_i,B_i)\ldotp
\]
Note that $\Aut(X_i)$ acts freely on $\Hom_{\sf{G}_i}(X_i\oplus A_i,B_i)$: either $\Hom_{\sf{G}_i}(X_i\oplus A_i,B_i)=\emptyset$, in which this is vacuous, or there is an isomorphism $X_i\oplus A_i\cong B_i$ and, in this latter case, if $f\in \Hom_{\sf{G}_i}(X_i\oplus A_i,B_i)$ and  $f(g\times \id)=f$. Then applying inverses, we must conclude that $g\times \id=\id$, and it must be that $g=\id$. It follows that the functor
\[
\Hom_{\sf{G}_i}(- \oplus A_i, B_i) : B\Aut(X_i) \longrightarrow \sf{Top},
\]
is projective cofibrant\textemdash that is, a cofibrant $\Aut(X_i)$-space\textemdash and, hence, this colimit is a homotopy colimit.

Now, the functor $\alpha$ defines for each $A,B \in \sf{G}_1$ an equivalence
\[
\alpha : \Hom_{\sf{G}_1}(A,B) \longrightarrow \Hom_{\sf{G}_2}(\alpha(A), \alpha(B)),
\]
and it follows that $\alpha$ induces a Dwyer-Kan equivalence $B\Aut(X)\to B\Aut(\alpha(X))$. This shows that $\langle \alpha \rangle$ is fully faithful. In other words, the induced map
\[
\underset{B\Aut(X)}{\colim}~\Hom_{\sf{G}_1}(X \oplus A, B) \xrightarrow{\ \simeq \ } \underset{B\Aut(\alpha(X))}{\colim}~\Hom_{\sf{G}_2}(\alpha(X)\oplus \alpha(A), \alpha(B)),
\]
is an equivalence. As for essential surjectivity of $\langle \alpha \rangle$, this follows since $\alpha$ is necessarily essentially surjective in the ordinary sense, the objects of $\sf{G}_i$ coincide with those of $\langle \sf{G}_i, \sf{G}_i \rangle$ for $i\in \{1,2\}$, and $\langle \alpha \rangle (X) = \alpha(X)$ for every $X \in \sf{G}_1$. 
\end{proof}

\begin{rem}
Let $\sf{G}$ be a topologically enriched monoidal groupoid satisfying the hypotheses of the situation considered above. While $\Hom_{\sf{G}}(-\oplus A,B)$ need not be projective cofibrant, if $Y$ is such that $\Hom(Y\oplus A,B)\ne\emptyset$, then, with the notation as above, it turns out that there are isomorphisms in the homotopy category of spaces
\[
\hocolim\limits_{X\in \sf{G}}\Hom_{\sf{G}}(X\oplus A,B)\cong \hocolim\limits_{X\in \sf{G}(Y)}\Hom_{\sf{G}}(X\oplus A,B)\cong \colim\limits_{B\Aut(Y)}\Hom_{\sf{G}}(Y\oplus A,B)\ldotp
\]
The first isomorphism appearing occurs since for $X\notin \sf{G}(Y)$, $\Hom_{\sf{G}}(X\oplus A,B)=\emptyset$\textemdash in terms of the classifying fibration, the total space is empty over the components of $\sf{G}(Z)$ with $Z\not\cong Y$. The second isomorphism follows since the inclusion $B\Aut(Y)\to\sf{G}(Y)$ is a Dwyer-Kan equivalence and the colimit appearing is the homotopy colimit, as we argued previously.
\end{rem}

\begin{cor}\label{cor: equivalence groupids equivalent functors}
Let $\sf{G}_1$ and $\sf{G}_2$ satisfy Hypothesis~\ref{hyp: groupoid}. If there is a monoidal Dwyer-Kan equivalence $\alpha : \sf{G}_1 \to \sf{G}_2$, then the induced adjunction
\[
\adjunction{\langle\alpha \rangle_!}{\Fun(\langle \sf{G}_1, \sf{G}_1 \rangle, \T)}{\Fun(\langle \sf{G}_2, \sf{G}_2 \rangle, \T)}{\langle\alpha\rangle^\ast},
\]
is an equivalence of $\infty$-categories.
\end{cor}
\begin{proof}
The map $\langle \alpha \rangle : \langle \sf{G}_1, \sf{G}_1\rangle \to \langle \sf{G}_2, \sf{G}_2\rangle$ is a Dwyer-Kan equivalence by Lemma~\ref{lem: equiv detected on groupoid cores}. The result follows immediately from Lemma~\ref{lem: equiv of diagrams induces equiv of functors}.
\end{proof}

\subsection{Generalized polynomial functors} We introduce a slight generalization of polynomial functors. This generalization recovers the standard definition in the cases we are interested in, but allows for a slick comparison between polynomial functors in different calculi. Observe as in~\cite[p.3]{Grayson} or~\cite[Definition 1.3]{RWW} that the monoidal unit $0$ in an initial object in $\langle \sf{G}, \sf{G} \rangle$ and hence for any $X \in \langle \sf{G}, \sf{G} \rangle$, the unique map $0 \to X$ is initial in $(\langle \sf{G}, \sf{G} \rangle)_{/X}$.

\begin{definition}\label{def: polynomial}
Let $\sf{G}$ be a topologically enriched monoidal groupoid and let $X \in \sf{G}$. A functor $F: \langle \sf{G}, \sf{G} \rangle \to \T$ is \emph{$X$-polynomial} if for each $V \in \es{J}$, the canonical map
\[
F(V) \longrightarrow \underset{U \in (\langle \sf{G}, \sf{G} \rangle_{/X})_0}{\holim}~F(V \oplus U),
\]
is an equivalence, where $(\langle \sf{G}, \sf{G} \rangle_{/X})_0$ is the slice category $(\langle \sf{G}, \sf{G} \rangle_{/X})$ with initial object removed.
\end{definition}

\begin{rem}
The above (homotopy) limit is the $\infty$-categorical limit taken over the \emph{space} $\langle \sf{G}, \sf{G} \rangle_{/X}$. We will model the $\infty$-categorical limit as a homotopy limit in the following way. The category $\langle \sf{G}, \sf{G} \rangle$ is topologically enriched and hence the category $\langle \sf{G}, \sf{G} \rangle_{/X}$ is a topologically internal category i.e., consists of a space of objects and a space of morphisms which satisfy the standard category axioms through continuous maps of spaces. A functor from a topologically internal category to spaces is the data of a space over the space of objects together with an action (over the space of objects) of the space of morphisms. In particular every functor $F: \langle \sf{G}, \sf{G} \rangle \to \mathsf{Top}$ has a canonical restriction to a functor $F: \langle \sf{G}, \sf{G} \rangle_{/X} \to \mathsf{Top}$. The topologically internal nerve $N(\langle \sf{G}, \sf{G} \rangle_{/X})$ of $\langle \sf{G}, \sf{G} \rangle_{/X}$ is a (non-Reedy fibrant) complete Segal space and the topologically internal Grothendieck construction allows us to model the functor $F(V \oplus -): \langle \sf{G}, \sf{G} \rangle_{/X} \to \mathsf{Top}$ as a left fibration of complete Segal spaces $\ms{F}_V \to N(\langle \sf{G}, \sf{G} \rangle_{/X})$. The homotopy limit of $F(V \oplus -)$ (when restricted to exclude the initial object) may then be modelled as the (derived) space of sections of this left fibration (suitably restricted to exclude the initial object).

This has been considered in the simplicial setting by Boavida de Brito~\cite{BdB} in which they make precise the statement that this models the $\infty$-categorical limit and that (up to Reedy fibrant replacement) the projective model structure on diagrams indexed by internal categories provides a model for the relevant $\infty$-category of functors. We make these ideas precise in the topological setting in~\cite{CarrTaggartcolimits}. 

We believe that the model here is equivalent to the model used by Weiss in the original constructions of orthogonal calculus~\cite{Weiss, WeissErratum}, and also note that this should be equivalent to the $\infty$-categorical limit based on quasicategories over the quasicategory $(N^h(\langle \sf{G}, \sf{G} \rangle)_{/X})_0$, where $N^h(-)$ denotes the homotopy coherent nerve. We do not pursue these equivalent descriptions here but note that our model has the correct universal property of a homotopy limit as equivalent to the derived space of maps from the constant functor to $F(V\oplus -)$,  see for example,~\cite{BdB, CarrTaggartcolimits}.
\end{rem}

\begin{rem}
The condition that a functor $F$ is $X$-polynomial is equivalent to the statement that for each $V$, the canonical map 
\[
F(V) \longrightarrow \Map^h_{N((\langle \sf{G}, \sf{G} \rangle_{/X})_0)}(N(\langle \sf{G}, \sf{G} \rangle_{/X})_0), \ms{F}_V),
\]
is an equivalence, where the right-hand side denotes the derived space of sections of the left fibration $\mathscr{F}_V \to N((\langle \sf{G}, \sf{G} \rangle_{/X})_0)$ of complete Segal spaces. By adjunction this map is equivalently described as the map
\[
F(V) \times N(\langle \sf{G}, \sf{G} \rangle_{/X})_0) \longrightarrow \mathscr{F}_V,
\]
over $N(\langle \sf{G}, \sf{G} \rangle_{/X})_0)$ given by
\[
(x, U) \longmapsto i_U(x)
\]
where $(x, U) \in F(V) \times N((\langle \sf{G}, \sf{G} \rangle_{/X})_0)$, and $i_U : F(V) \to F(V\oplus U)$ is the map induced by the canonical inclusion $V \hookrightarrow V\oplus U$.
\end{rem}

\begin{rem}
For orthogonal calculus, i.e., $\sf{G} = \sf{G}^\mbf{O}$ (compare Example~\ref{exs: groupoid cores and S-1S}), the choice $X = \R^{n+1}$ recovers the notion of $n$-polynomial. Note also that this is invariant under isomorphism in $\sf{G}$. For instance, if $V$ is any $(n+1)$-dimensional real inner product space then $F$ is $V$-polynomial if and only if $F$ is $n$-polynomial. 
\end{rem}

Let $\es{C}$ be a topologically internal category with space of objects $\sf{ob}(\es{C})$ and space of morphisms $\mor(\es{C})$. The $q$-simplicies of the topological nerve of $\es{C}$ are given by
\[
N_q\es{C} = \mor(\es{C}) \times_{\sf{ob}(\es{C})} \cdots \times_{\sf{ob}(\es{C})} \mor(\es{C}),
\]
with $\mor(\es{C})$ appearing $q$-times. This defines a functor 
\[
N : \sf{Cat}(\sf{Top}) \longrightarrow \sf{Top}^{\Delta^\op},
\]
    from the category of topologically internal categories to the category of simplicial spaces. Under mild hypothesis on the internal categories, satisfied in all our examples, the internal nerve is a complete Segal space and hence may be viewed as an $\infty$-category.

\begin{lem}\label{lem: DK equiv on slices}
Let $\es{J}_1$ and $\es{J}_2$ be small topologically enriched categories. If there is a Dwyer-Kan equivalence $\alpha: \es{J}_1 \to \es{J}_2$, then for every $A \in \es{J}_1$, the induced functor
\[
N(\alpha): N((\es{J}_1)_{/A})\longrightarrow N((\es{J}_2)_{/\alpha(A)}),
\]
is an equivalence of simplicial spaces.
\end{lem}
\begin{proof}
The slice category $(\es{J}_1)_{/A}$ is an internal category with space of objects consisting of maps with target $A$, and space of morphisms given by flags of length two in $\es{J}_1$ with target $A$, i.e., a point in the space of morphisms of $(\es{J}_1)_{/A}$ is a sequence of maps $(V_0 \to V_1 \to A)$ where $V_0, V_1 \in \es{J}_1$. In general a point in a $q$-simplex of the internal nerve $N((\es{J}_1)_{/A})$ is a flag of length $q+1$ in $\es{J}_1$, i.e., a sequence
\[
(V_0 \to V_1 \to \cdots \to V_q \to A),
\]
of morphisms in $\es{J}_1$. For $0$-simplicies of the nerve, the induced map is an equivalence since $\alpha$ is a Dwyer-Kan equivalence, i.e., since $\alpha$ is fully faithful. A similar argument on higher simplicies, using the right properness of $\sf{Top}$ implies that the induced map 
\[
N(\alpha): N((\es{J}_1)_{/A}) \longrightarrow N((\es{J}_2)_{/\alpha(A)}),
\]
is a levelwise equivalence, hence, an equivalence of simplicial spaces.
\end{proof}

\begin{prop}\label{prop: generalized polynomial comparision} \label{prop: symplectic n-poly iff unitary n-poly}
Let $\sf{G}_1$ and $\sf{G}_2$ satisfy Hypothesis~\ref{hyp: groupoid}. If $\alpha: \sf{G}_1 \to \sf{G}_2$ a  monoidal Dwyer-Kan equivalence. and $X \in \langle \sf{G}_1, \sf{G}_1\rangle$, then a functor $F: \langle \sf{G}_2, \sf{G}_2 \rangle \to \T$ is $\alpha(X)$-polynomial if and only if $\alpha^\ast F: \langle \sf{G}_1, \sf{G}_1\rangle \to \T$ is $X$-polynomial.
\end{prop}
\begin{proof}
Assume that $F$ is $\alpha(X)$-polynomial. We want to show that for every $V \in \langle \sf{G}_1, \sf{G}_1\rangle$, the functor
\[
(N\langle\alpha\rangle^\ast F)(V \oplus -): N((\langle \sf{G}_1, \sf{G}_1\rangle_{/X})_0) \longrightarrow \T, 
\]
is a limit diagram. Transporting the problem to left fibrations, the left fibration associated to $(N\langle\alpha\rangle^\ast F)(V \oplus -)$ is naturally equivalent to the pullback along
\[
N\langle\alpha\rangle : N((\langle \sf{G}_1, \sf{G}_1\rangle_{/X})_0)  \longrightarrow N((\langle \sf{G}_2, \sf{G}_2\rangle_{/X})_0),
\]
of the left fibration $\ms{F}_{\alpha(V)}$ which corresponds to to the functor 
\[
F(\alpha(V) \oplus -): N((\langle \sf{G}_2, \sf{G}_2\rangle_{/X})_0)  \to \T.
\]
By Lemma~\ref{lem: DK equiv on slices} the map $\alpha$ is an equivalence, and hence the fibration $N\langle\alpha\rangle^\ast \ms{F}_{\alpha(V)}$ is equivalent to the fibration $\ms{F}_{\alpha(V)}$. It follows that the spaces of sections of these fibrations are equivalent, and hence the canonical map
\[
F(\alpha(V)) \longrightarrow \Map_{N((\langle \sf{G}_2, \sf{G}_2\rangle_{/\alpha(X)})_0)}(N((\langle \sf{G}_2, \sf{G}_2\rangle_{/\alpha(X)})_0), \ms{F}_{\alpha(V)}),
\]
being an equivalence implies that the canonical map
\[
(N\langle\alpha\rangle^\ast F)(V) = F(\alpha(V)) \longrightarrow \Map_{N((\langle \sf{G}_1, \sf{G}_1\rangle_{/X})_0)}(N((\langle \sf{G}_1, \sf{G}_1\rangle_{/X})_0), N\langle \alpha \rangle^\ast\ms{F}_{\alpha(V)}),
\]
is an equivalence, and hence $F$ is $X$-polynomial. The converse follows by carefully reading this argument in reverse.
\end{proof}

For ease of notation we will no longer distinguish between $N\langle \alpha \rangle$ and $\alpha$.

\begin{cor}\label{cor: compare poly with inverse}
Let $\sf{G}_1$ and $\sf{G}_2$ satisfy Hypothesis~\ref{hyp: groupoid}. If $\alpha: \sf{G}_1 \to \sf{G}_2$ is a monoidal Dwyer-Kan equivalence, and $X \in \es{J}_1$, then a functor $F: \langle \sf{G}_1, \sf{G}_1 \rangle \to \T$ is $X$-polynomial if and only if $\alpha_! F: \langle \sf{G}_2, \sf{G}_2\rangle \to \T$ is $\alpha(X)$-polynomial. 
\end{cor}
\begin{proof}
Since the adjunction
\[
\adjunction{\alpha_!}{\Fun(\langle \sf{G}_1, \sf{G}_1 \rangle, \T)}{\Fun(\langle \sf{G}_2, \sf{G}_2 \rangle, \T)}{\alpha^\ast},
\]
is an equivalence of $\infty$-categories, both the unit and counit are equivalences. It follows that a functor $F$ is $X$-polynomial if and only if $\alpha^\ast\alpha_!(F)$ is $X$-polynomial. By Proposition~\ref{prop: generalized polynomial comparision}, $\alpha^\ast\alpha_!(F)$ is $X$-polynomial if and only if $\alpha_!(F)$ is $\alpha(X)$-polynomial. 
\end{proof}

Denote by $\mathsf{Poly}^X(\langle \sf{G}, \sf{G} \rangle, \T)$ the full sub-$\infty$-category of $\Fun(\langle G, G \rangle, \T)$ spanned by the $X$-polynomial functors. Proposition~\ref{prop: generalized polynomial comparision} implies the following.

\begin{cor}\label{cor:equiv_poly}
Let $\sf{G}_1$ and $\sf{G}_2$ satisfy Hypothesis~\ref{hyp: groupoid}. If $\alpha: \sf{G}_1 \to \sf{G}_2$ is a monoidal Dwyer-Kan equivalence, and $X \in \es{J}_1$, then the adjoint pair
\[
\adjunction{\alpha_!}{\mathsf{Poly}^X(\langle \sf{G}_1, \sf{G}_1 \rangle, \T)}{\mathsf{Poly}^{\alpha(X)}(\langle \sf{G}_2, \sf{G}_2 \rangle, \T)}{\alpha^\ast},
\]
is an equivalence of $\infty$-categories.
\end{cor}
\begin{proof}
By Corollary~\ref{cor: equivalence groupids equivalent functors}, the adjunction 
\[
\adjunction{\alpha_!}{\Fun(\langle \sf{G}_1, \sf{G}_1 \rangle, \T)}{\Fun(\langle \sf{G}_2, \sf{G}_2 \rangle, \T)}{\alpha^\ast},
\]
is an equivalence of $\infty$-categories. Since $\mathsf{Poly}^X(\langle \sf{G}_1, \sf{G}_1 \rangle, \T)$ is defined as a full sub-$\infty$-category, the functor $\alpha^\ast$ is fully faithful, hence it is left to show that $\alpha^\ast$ is essentially surjective, but this follows from essential surjectivity on the level of functors and Proposition~\ref{prop: generalized polynomial comparision} and Corollary~\ref{cor: compare poly with inverse}.
\end{proof}


\section{Symplectic Weiss calculus and the deformation retraction}\label{section: symplectic calculus}

In this section, we provide a candidate for symplectic Weiss calculus and show that it \emph{deformation retracts} onto unitary calculus. By utilising Section~\ref{sec: groupoid cores}, the deformation retract becomes an exercise in investigating the groupoid cores of the homogeneous categories used to construct unitary and symplectic calculus, and hence essentially a commentary on the unitary group $U(n)$ being the maximal compact subgroup of the symplectic group $\mathsf{Sp}(2n, \R)$.

\subsection{The unitary and symplectic groupoids}
Denote by $\sf{G}^\mbf{S}$ the topologically enriched groupoid of symplectic vector spaces and symplectic isomorphisms\textemdash with the operation on forms given by direct sum\textemdash  and denote by $\es{J}^\mbf{S}$ Qullien's bracket construction on $\sf{G}^\mbf{S}$, as in Example~\ref{exs: groupoid cores and S-1S}. There is a forgetful functor $\Im: \es{J}^\mbf{U} \to \es{J}^\mbf{S}$ which sends a Hermitian inner product space $V$ to its underlying real vector space $rV$ with symplectic structure given by the imaginary part of the Hermitian inner product, $\omega_{rV} = \Im(\langle-\mid-\rangle)$. A routine linear algebra exercise shows that this gives a well-defined symplectic form on $rV$. 

\begin{lem}\label{lem: unitary and symplectic equivalence of diagram cats}
The functor $\Im : \es{J}^\mbf{U} \to \es{J}^\mbf{S}$ is a Dwyer-Kan equivalence.
\end{lem}
\begin{proof}
First note that $\Im: \es{J}^\mbf{U} \to \es{J}^\mbf{S}$ is equivalently described as 
\[
\langle \Im \rangle : \langle \sf{G}^\mbf{U}, \sf{G}^\mbf{U} \rangle \longrightarrow \langle \sf{G}^\mbf{S}, \sf{G}^\mbf{S} \rangle\ldotp
\]
By Lemma~\ref{lem: equiv detected on groupoid cores} it suffices to show that the functor 
\[
\Im : \sf{G}^\mbf{U} \longrightarrow \sf{G}^\mbf{S},
\]
is a monoidal Dwyer-Kan equivalence. It is easy to see that $\Im$ is essentially surjective and monoidal, so we only need to verify that $\Im$ is fully faithful. For this, consider $\mbf{C}^n\in\sf{G}^{\mbf{U}}$. On these objects, $\Im$ induces an embedding
\[
U(n)\to \Aut_{\mbf{S}}(r\mbf{C}^n)\cong \mathsf{Sp}(2n,\mbf{R})
\]
compatible with the group structures. Since $U(n)$ is a maximal compact subgroup of the connected Lie group $\Aut_{\mbf{S}}(r\mbf{C}^n)\cong\mathsf{Sp}(2n,\mbf{R})$, there is a deformation retract of $\Aut_{\mbf{S}}(r\mbf{C}^n)$ onto $U(n)$ and this inclusion is therefore a weak equivalence. This implies that $\Im$ is fully faithful and therefore a Dwyer-Kan equivalence.
\end{proof}

\begin{rem}
The above result implies that complex Stiefel manifolds are homotopy equivalent to symplectic Stiefel manifolds. Results of this nature have been known for some time, see for example, \cite{AjayiBanyaga} and the references therein. In fact, in the case of semisimple connected Lie group results of this kind date back to Cartan, see also the book of Helgason~\cite[Ch. VI, Theorem 2.2]{Helgason}.
\end{rem}

\subsection{Symplectic Functors and polynomial approximation}
The equivalence of groupoids between $\sf{G}^\mbf{U}$ and $\sf{G}^\mbf{S}$ allows one to leverage unitary calculus arguments to show that an $n$-polynomial symplectic functor is $(n+1)$-polynomial.

\begin{prop}\label{prop: symplectic n-poly implies n+1-poly}
If a symplectic functor $F: \es{J}^\mbf{S} \to \T$ is $n$-polynomial, then it is $(n+1)$-polynomial.
\end{prop}
\begin{proof}
If $F$ is $n$-polynomial, then by Proposition~\ref{prop: generalized polynomial comparision} the unitary functor $\Im^\ast F$ is $n$-polynomial, hence $(n+1)$-polynomial by the corresponding result in unitary calculus, see for example,~\cite[Proposition 3.10]{TaggartUnitary}. An application of Corollary~\ref{cor: compare poly with inverse} yields that the symplectic functor $F \simeq \Im_! \Im^\ast F$ is $(n+1)$-polynomial, whence the result.
\end{proof}

Let $\es{J}$ be one of the categories of vector spaces, and denote by $\mbf{k}$ the underlying field. For any functor $F : \es{J} \to \T$, define $\tau_nF : \es{J} \to \T$ to be the functor given by
\[
\tau_nF(V) := \holim_{U \in (\es{J} _{/\mbf{k}^{n+1}})_0} F(V \oplus U).
\]

This functor is in general not $n$-polynomial as it need not be idempotent but iterating to obtain a functor,
\[
T_nF(V) := \hocolim_{k \geq 0} (\tau_n)^kF(V),
\]
typically yields an $n$-polynomial functor. In all known cases of Weiss calculus, this functor is the universal $n$-polynomial functor under $F$. The proof that $T_nF$ is $n$-polynomial is one of the most complex aspects of setting up a Weiss calculus, see for example,~\cite{Weiss, WeissErratum}. It involves an intricate analysis of the (homotopy) (co)limits in the definition of $T_nF$ together with the fact that the slice category $(\es{J} _{/\mbf{k}^{n+1}})_0$ is a \emph{topologically internal category}, or rather, the nerve of $(\es{J} _{/\mbf{k}^{n+1}})_0$ is a (non-Reedy fibrant) complete Segal space. We can utilise Proposition~\ref{prop: symplectic n-poly iff unitary n-poly} to avoid this analysis in the symplectic setting.

\begin{prop}\label{prop: Tn n-poly}
    For a symplectic functor $F: \es{J}^\mbf{S} \to \T$, the functor $T_nF$ is $n$-polynomial. 
\end{prop}
\begin{proof}
By Proposition~\ref{prop: symplectic n-poly iff unitary n-poly} it suffices to provide an equivalence between $\Im^\ast(T_nF)$ and $T_n^\mbf{U}(\Im^\ast F)$ where the latter denotes the universal $n$-polynomial approximation in unitary calculus. Since the functor $\Im^\ast$ commutes with all limits and colimits it suffices to provide an identification between the relevant homotopy limits, but this is precisely the content of the proof of Proposition~\ref{prop: symplectic n-poly iff unitary n-poly}.
\end{proof}

\begin{prop}\label{prop: Tn universal}
If $F$ is a symplectic functor, then $T_nF$ is the universal $n$-polynomial functor under $F$.  
\end{prop}
\begin{proof}
The proof follows~\cite[Theorem 6.3]{Weiss}, but noting that the errata~\cite{WeissErratum} only corrects the proof that $T_nF$ is $n$-polynomial, see also~\cite[Corollary 3.2.0.10.]{Hendrian}.
\end{proof}

\begin{rem}
The classical proof that $T_nF$ is $n$-polynomial~\cite{WeissErratum} uses Weiss's chosen model for the homotopy limit of a functor $F: (\es{J}_{/\mbf{R}^{n+1}})_0 \to \mathsf{Top}$. In all other models of Weiss calculus, the proof uses an analogous model for the homotopy limit. Examining the content of~\cite[Theorem 6.3]{Weiss} and~\cite{WeissErratum} one observes that having a model is key only in proving~\cite[Lemma e.3]{WeissErratum}, i.e., the statement that if $f: E \to F$ is such that $f(V): E(V) \to F(V)$ is approximately $((n+1)\dim(V))$-connected, then $\tau_nf(V)$ is more connected. Using the model for the homotopy limit as the derived space of sections of a left fibration of complete Segal spaces together with the formula for connectivity of natural transformation spaces, see for example,~\cite[Proposition A.1.1]{DottoEquivariantdiagrams}, one can readily prove the same connectivity bound as in~\cite[Lemma e.3]{WeissErratum} by writing the relevant (derived) section spaces as (homotopy) fibers.
\end{rem}

\subsection{The deformation retract of calculi}
Precomposition with $\Im$ induces an adjunction
\begin{equation*}\label{eq: Unitary symplectic adjunction}
\adjunction{\Im_!}{\Fun(\es{J}^{\mbf{U}}, \T)}{\Fun(\es{J}^\mbf{S}, \T)}{\Im^\ast},
\end{equation*}
which we will use to compare the calculi. An application of Lemma~\ref{lem: equiv of diagrams induces equiv of functors}, yields the following result. 

\begin{lem}
The adjunction
\[
\adjunction{\Im_!}{\Fun(\es{J}^{\mbf{U}}, \T)}{\Fun(\es{J}^\mbf{S}, \T)}{\Im^\ast},
\]
is an equivalence of $\infty$-categories.
\end{lem}

Consider the full subcategory $\poly{n}(\es{J}^\mbf{S}, \T)$ of $n$-polynomial functors. The adjunction of Lemma \ref{lem: unitary and symplectic equivalence of diagram cats} descends to an adjunction between the categories of $n$-polynomial functors 
\[
\adjunction{\Im_!}{\poly{n}(\es{J}^\mbf{U}, \T)}{\poly{n}(\es{J}^\mbf{S}, \T)}{\Im^\ast},
\]
which is an equivalence of $\infty$-categories by Corollary~\ref{cor:equiv_poly}.

\begin{lem}
The adjoint pair
\[
\adjunction{\Im_!}{\poly{n}(\es{J}^\mbf{U}, \T)}{\poly{n}(\es{J}^\mbf{S}, \T)}{\Im^\ast},
\]
is an equivalence of $\infty$-categories.
\end{lem}

Proposition~\ref{prop: symplectic n-poly implies n+1-poly} and Proposition~\ref{prop: Tn universal} imply that for each $n\geq 1$, there is natural transformation
\[
T_{n}F \longrightarrow T_{n-1}F
\]
which assemble into a Weiss tower
\[\begin{tikzcd}
	&& F \\
	\cdots & {T_nF} & \cdots & {T_1F} & {T_0F}
	\arrow[from=2-4, to=2-5]
	\arrow[from=2-2, to=2-3]
	\arrow[from=2-3, to=2-4]
	\arrow[bend right=30, from=1-3, to=2-2]
	\arrow[from=2-1, to=2-2]
	\arrow[bend left=30, from=1-3, to=2-4]
	\arrow[bend left=20, from=1-3, to=2-5]
\end{tikzcd}\]
under $F$. We may represent this Weiss tower as a functor
\[
\mathsf{Tow}(F) : \mbf{Z}_{\geq 0}^\op \longrightarrow \Fun(\es{J}, \T), \qquad n \longmapsto T_nF,
\]
and hence globally as a functor
\[
\mathsf{Tow}: \Fun(\es{J}, \T) \longrightarrow \Fun(\mbf{Z}_{\geq 0}^\op, \Fun(\es{J},\T)),
\]
which sends $F$ to $\mathsf{Tow}(F)$.

\begin{thm}
For any symplectic functor $F: \es{J}^\mbf{S} \to \T$, there is an equivalence
\[
\mathsf{Tow}^\mbf{S}(F) \simeq \Im_!\mathsf{Tow}^\mbf{U}(\Im^\ast F).
\]
\end{thm}
\begin{proof}
Since $\Im_!$ is inverse to $\Im^\ast$ at the level of $\infty$-categories, it suffices exhibit an equivalence 
\[
\Im^\ast \mathsf{Tow}^\mbf{S}(F) \simeq \mathsf{Tow}^\mbf{U}(\Im^\ast F). 
\]
Recall that there is a map $F\to T_n F$ initial among all $n$-polynomial functors. Since $\Im^\ast$ is an equivalence, it follows that for each $n\ge0$, $\Im^\ast(F)\to \Im^\ast (T_k F)$ is likewise initial. Thus, for each $n\ge 1$, in the following square
\[\begin{tikzcd}
	{\Im^\ast (T_nF)} & {\Im^\ast(T_{n-1}F)} \\
	{T_n^\mbf{U}(\Im^\ast F)} & {T_{n-1}^\mbf{U}(\Im^\ast F)}
	\arrow[from=1-1, to=1-2]
 	\arrow[from=1-2, to=2-2]
 	\arrow[from=1-1, to=2-1]
 	\arrow[from=2-1, to=2-2]
 \end{tikzcd}
\]
all composites are maps under $\Im^\ast F$, where the vertical maps are the equivalences of Proposition~\ref{prop: Tn n-poly} and horizontal arrows arise from the Weiss tower. Hence, this square commutes up to homotopy. This implies that there exists a natural equivalence $\Im^\ast\sf{Tow}^\mbf{S}(F)\to \sf{Tow}^{\mbf{U}}(\Im^\ast F)$.
\end{proof}
\begin{rem}
Here we have only considered functors to unpointed spaces. To obtain analogous results for functors to pointed spaces requires little more than the observation that a functor $F: \es{J} \to \T_\ast$ is $n$-polynomial if and only if the induced functor $F: \es{J} \to \T$ is $n$-polynomial, i.e., that limits in pointed spaces are computed in spaces. For more on the interplay between pointed and unpointed orthogonal calculus in the $\infty$-categorical setup see for example,~\cite{Hendrian}.
\end{rem}

\part{Quarternion Weiss calculus}\label{part: quaternion calc}
In this part, we construct (in the now standard way) quaternion calculus, i.e., a version of Weiss calculus built on the compact symplectic groups $\mathsf{Sp}(n) =\Aut(\mbf{H}^n)$, and provide comparisons between this calculus and the other known versions of calculus. This construction is somewhat standard as the quaternions support just enough linear algebra.

\section{Quaternion Steifel combinatorics}\label{sec: combinatorics}
Much of orthogonal calculus is governed by understanding the geometry of Stiefel manifolds, see for example,~\cite[\S1, \S4]{Weiss}. For example, the existence of a universal $n$-polynomial approximation, and the existence of the Weiss tower itself are reliant on understanding the geometry of Stiefel manifolds. In fact, this `Stiefel combinatorics' is still an essential input to constructing orthogonal calculus using the topos-theoretic machine\footnote{The exact proof that their topos-theoretic construction recovers orthogonal calculus has not yet appeared in the literature, but is claimed in~\cite{ABFJ}.} of Anel, Biedermann, Finster and Joyal~\cite{ABFJ}, and similar analysis of Stiefel manifolds is required in the construction of all known variants of Weiss calculus. In this section, we introduce the necessary combinatorics with Stiefel manifolds to construct a quaternion version of Weiss calculus. For the reader interested less in the finer details of the constructions and more in the properties of the calculus or those familiar with~\cite[\S1, \S4]{Weiss}, we invite you to skip this section.

\subsection{Quaternion vector spaces}
Denote by $\mbf{H}$ the division algebra of the quaternions with inner product given by 
\[
\langle-,-\rangle : \mbf{H} \times \mbf{H} \longrightarrow \mbf{H}, (q_1,q_2) \longmapsto q_1\overline{q_2}. 
\]
By a \emph{quaternion vector space} we mean a right $\mbf{H}$-module. Denote by $\es{J}^\mbf{H}$ the category of finite-dimensional quaternion inner product subspaces of $\mbf{H}^\infty$ and inner product preserving embeddings. The space of maps $\es{J}^\mbf{H}(V,W)$ may be topologised as the Stiefel manifold of $\dim_{\mbf{H}}(V)$-frames in $W$, i.e., as the homogeneous space 
\[
\sf{Sp}(W)/\sf{Sp}(V^\perp),
\]
where $\sf{Sp}(V)$ is the \emph{compact Symplectic group}, or equivalently the space $\es{J}^\mbf{H}(V,V)$. In other words, the category $\es{J}^\mbf{H}$, is Quillen's bracket construction on the groupoid given by the disjoint union of the topological groups $\mathsf{Sp}(V)$, compare with Example~\ref{exs: groupoid cores and S-1S}.

For each $n \in \mbf{N}$, sitting over the space of linear isometries $\es{J}^\mbf{H}(V,W)$ is the $n$-fold orthogonal complement vector bundle $\gamma^\mbf{H}_n(V,W)$ with fiber over a linear isometry $f$ given by $\mbf{H}^n \otimes f(V)^\perp$. To ease notation we write $nV$ for the tensor product $\mbf{H}^n \otimes V$. Denote by $\es{J}_n^\mbf{H}$ the category with the same objects as $\es{J}^\mbf{H}$ and morphism space given by the Thom space of the total space of $\gamma^\mbf{H}_n(V,W)$.

\subsection{Stiefel Combinatorics over the quaternions}
As is standard in all other forms of orthogonal calculus, see for example, \cite{Weiss, TaggartUnitary, TaggartReality}, the sphere bundle of the vector bundle $\gamma_n(V,W)$ may be constructed (up to homeomorphism) as a certain homotopy colimit of morphism spaces in $\es{J}^\mbf{H}$. The proof is a rather involved interaction between homotopy theory and linear algebra, the first account of which was provided by Weiss in \cite[Proposition 4.2]{Weiss}.

\begin{rem}
The following proof involves heavy linear algebra. Due to the non-commutativity of the quaternions care must be taken when discussing eigenvalues and eigenvectors of $\mbf{H}$-linear endomorphisms. For instance, given an eigenvalue $\lambda$ of a $\mbf{H}$-linear endomorphism on a quaternion vector space $V$ the eigenspace $E(\lambda)$ need only be a real subspace of $V$. If the eigenvalue $\lambda$ is real then the eigenspace is a $\mbf{H}$-linear subspace of $V$. We will only need to consider real eigenvalues since the eigenvalues of any self-adjoint linear isometry on a quaternion vector space $V$ are necessarily real. A good review of linear algebra over the quaternions can be found in~\cite[Appendix B]{SchwedeGlobalSplittings}.
\end{rem}

\begin{prop}
For each $n \in \mbf{N}$, there is a homeomorphism
\[
\underset{0 \neq U \subseteq \mbf{H}^{n+1}}{\hocolim}~\es{J}^\mbf{H}(U \oplus V, W) \longrightarrow S\gamma^\mbf{H}_{n+1}(V,W),
\]
which is natural in $V,W \in \es{J}^\mbf{H}$. 
\end{prop}
\begin{proof}
The proof proceeds as in~\cite[Theorem 4.1]{Weiss} and~\cite[Theorem 4.1]{TaggartUnitary} taking into account the delicate nature of linear algebra over the quaternions as discussed in the preceding remark, and choosing a suitable model for the colimit as a homotopy colimit.
\end{proof}

\begin{rem}
The model for the homotopy colimit used by Weiss in~\cite{Weiss} is equivalent to the model based on left fibrations of (complete) Segal spaces in the sense of Boavida de Brito~\cite{BdB}. By~\cite[Theorem 1.22 and Proposition 4.1]{BdB}, it follows that the homotopy colimit of a left fibration is the geometric realization of the total simplicial space, and this models the quasicategorical homotopy colimit of the associated fibration of quasicategories, see also for example,~\cite[\href{https://kerodon.net/tag/02VF}{Corollary 02VF}]{kerodon} and~\cite[Proposition 10.3.6(viii)]{RVEl}, from which it follows that the formula for such a geometric realization matches that used by Weiss. We discuss this further in~\cite{CarrTaggartcolimits}.
\end{rem}

The final results in this section are homotopy cofiber sequences relating the morphism spaces in the categories $\es{J}^\mbf{H}_n$ and $\es{J}^\mbf{H}_{n+1}$. The proof of the following result may be taken from \cite[Proposition 1.2]{Weiss} by appropriately replacing the real numbers by the quaternions. 

\begin{prop}
For each $n\in \mbf{N}$, and $V,W \in \es{J}^\mbf{H}$, there is a natural cofiber sequence
\[
\es{J}^\mbf{H}_n(V \oplus \mbf{H}, W) \wedge S^{4n} \longrightarrow \es{J}^\mbf{H}_n(V,W) \longrightarrow \es{J}^\mbf{H}_{n+1}(V,W),
\]
as functors $(\es{J}^\mbf{H}_n)^\op \times \es{J}^\mbf{H}_n \to \T_\ast$. 
\end{prop}

The morphism spaces in $\es{J}^\mbf{H}_{n+1}$ may also be described relative to the sphere bundle of the $(n+1)$-fold orthogonal complement bundle. Again, this is nothing special about the quaternion case, and holds in other versions of orthogonal calculus, see for example, \cite[\S5]{Weiss} and \cite[Lemma 5.5]{BarnesOman}. 

\begin{lem}
For each $n\in \mbf{N}$, and $V,W \in \es{J}^\mbf{H}$, there is a natural cofiber sequence
\[
S\gamma^\mbf{H}_{n+1}(V, W)_+ \longrightarrow \es{J}^\mbf{H}_0(V,W) \longrightarrow \es{J}^\mbf{H}_{n+1}(V,W),
\]
as functors $(\es{J}^\mbf{H}_0)^\op \times \es{J}^\mbf{H}_0 \to \T_\ast$. 
\end{lem}

\section{Polynomial functors and the Weiss tower}

Following Definition~\ref{def: polynomial}, a functor $F : \es{J}^\mbf{H} \to \T_\ast$ is polynomial of degree less than or equal $n$ if for each $V \in \es{J}^\mathbf{H}$, the canonical map
\[
F(V) \longrightarrow \holim_{U \in (\es{J}^\mbf{H}_{/\mbf{H}^{n+1}})_0} F(V \oplus U),
\]
is an equivalence. The category $(\es{J}^\mbf{H}_{/\mbf{H}^{n+1}})$ is equivalent to the poset $\es{P}(\mbf{H}^{n+1})$ of subspaces of $\mbf{H}^{n+1}$, hence a functor is $n$-polynomial if and only if the canonical map
\[
F(V) \longrightarrow \underset{U \in \es{P}_0(\mbf{H}^{n+1})}{\holim} F(V \oplus U),
\]
is an equivalence, matching the more familiar definition\footnote{Weiss~\cite{Weiss} and the second author~\cite{TaggartUnitary,TaggartReality} used homotopy limits indexed on topological categories, but this is equivalent to the $\infty$-categorical notion by~\cite{CarrTaggartcolimits} and~\cite{BdB}.} from orthogonal and unitary calculus~\cite[Definition 5.1]{Weiss}. We denote by $\poly{n}(\es{J}^\mbf{H}, \T_\ast)$ the full sub-$\infty$-category of $\Fun(\es{J}^\mbf{H}, \T_\ast)$ spanned by the $n$-polynomial functors.

\subsection{Polynomial approximation}As before we denote by $T_n$ the repeated iteration of the functor $\tau_n : \Fun(\es{J}^\mbf{H}, \T_\ast) \longrightarrow \Fun(\es{J}^\mbf{H}, \T_\ast)$,
given by 
\[
\tau_nF(V) = \underset{U \in (\es{J}^\mbf{H}_{/\mbf{H}^{n+1}})_0}{\holim}~F(V \oplus U). 
\]
In other words, 
\[
T_nF(V) = \hocolim_k (\tau_n)^kF(V). 
\]
In the case of quaternion Weiss calculus we can show that $T_n$ has the required universal property: $T_nF$ is the universal $n$-polynomial functor under $F$. For the $\infty$-categorical version of the following result in orthogonal calculus see for example,~\cite[Theorem 3.2.0.8 and Corollary 3.2.0.10]{Hendrian}.

\begin{prop}
    For each $n \geq 0$, the functor $T_n$ is the left exact left adjoint to the inclusion 
    \[
    \poly{n}(\es{J}^\mbf{H}, \T_\ast) \hookrightarrow \Fun(\es{J}^\mbf{H}, \T_\ast),
    \]
    of the full sub-$\infty$-category of $n$-polynomial functors.
\end{prop}
\begin{proof}
    This is equivalent to the statement that for each functor $F$, $T_nF$ is the universal $n$-polynomial functor under $F$. With the Stiefel combinatorics in place this follows readily from~\cite[Theorem 6.3]{Weiss} and \cite{WeissErratum} making the appropriate alternations from real linear algebra to linear algebra over the quaternions. 
\end{proof}

\subsection{The Weiss tower} We conclude this section by constructing the Weiss tower in quaternion calculus. We first record a fact on the relationship between $n$-polynomial functors and $(n+1)$-polynomial functors which ensures that the Weiss tower has the required properties.

\begin{lem}
For each $n \geq 0$, every $n$-polynomial functor is $(n+1)$-polynomial.
\end{lem}
\begin{proof}
    The proof follows that of~\cite[Proposition 5.4]{Weiss} replacing the real Stiefel combinatorics with the relevant quaternion version, see Section~\ref{sec: combinatorics}.
\end{proof}

The inclusion $\mbf{H}^n \hookrightarrow \mbf{H}^{n+1}$ induces a natural transformation $T_{n+1} \to T_{n}$, and hence a Weiss tower 
\[\begin{tikzcd}
	&& F \\
	\cdots & {T_nF} & \cdots & {T_1F} & {T_0F}
	\arrow[from=2-4, to=2-5]
	\arrow[from=2-2, to=2-3]
	\arrow[from=2-3, to=2-4]
	\arrow[bend right=30, from=1-3, to=2-2]
	\arrow[from=2-1, to=2-2]
	\arrow[bend left=30, from=1-3, to=2-4]
	\arrow[bend left=20, from=1-3, to=2-5]
\end{tikzcd}\]
under $F$. To complete our discussion on quaternion Weiss calculus we classify the layers of the tower, i.e., functors
\[
D_nF = \fiber(T_nF \longrightarrow T_{n-1}F).
\]
As with the other versions of Weiss calculus, the fact that $(n-1)$-polynomial functors are $n$-polynomial implies that $D_nF$ satisfies the following definition.

\begin{definition}
    A functor $F: \es{J}^\mbf{H} \to \T_\ast$, is \emph{homogeneous of degree $n$} or equivalently, $n$-homogeneous if $F$ is $n$-polynomial and $T_{n-1}F$ is trivial. We denote by $\homog{n}(\es{J}^\mbf{H}, \T_\ast)$ the full sub-$\infty$-category of $\Fun(\es{J}^\mbf{H}, \T_\ast)$ spanned by the $n$-homogeneous functors.
\end{definition}

\section{Derivatives and the classification of homogeneous functors}
To understand the homotopy type of a functor $F$ from the associated Weiss tower we will classify the layers of the Weiss tower in terms of spectra with an action of $\mathsf{Sp}(n)=\Aut(\mbf{H}^n)$, analogous to the orthogonal calculus classification of the layers of the orthogonal Weiss tower in terms spectra with an action of $O(n)=\Aut(\mbf{R}^n)$.

\subsection{Derivatives}
We now explain how to each functor $F: \es{J}^\mbf{H} \to \T_\ast$ one may associate a spectrum with an action of $\Aut(\mbf{H}^n) = \mathsf{Sp}(n)$, called the $n$-th derivative of $n$-th coefficient spectrum of $F$. We take a slightly different (but equivalent) approach than Weiss' original approach.

\begin{definition}\label{def: derivative}
Let $F: \es{J}^\mbf{H} \to \T_\ast$. The $n$-th derivative of $F$, denoted $F^{(n)}$ is the functor defined as
\[
F^{(n)}(V) := \fib ( F(V) \longrightarrow \tau_{n-1}F(V)).
\]
\end{definition}

This defines a functor $F^{(n)} : \es{J}^\mbf{H} \to \T_\ast$, and we now show that the $n$-th derivative is naturally a ``spectrum of multiplicity $4n$'' by constructing structure maps. 

\begin{prop}
The $n$-th derivative of $F: \es{J}^\mbf{H} \to \T_\ast$ has structure maps of the form
\[
S^{4n} \wedge F^{(n)}(V) \longrightarrow F^{(n)}(V \oplus \mbf{H}).
\]
\end{prop}
\begin{proof}
This is essentially a rephrasing of~\cite[Proposition 2.2]{Weiss}. By the quaternion version of~\cite[Proposition 5.3]{Weiss} we can write
\[
F^{(n)}(V) = \nat(\es{J}_n(V,-), F),
\]
using the quaternion Stiefel combinatorics of Section~\ref{sec: combinatorics}. In particular, the quaternion version of~\cite[Proposition 2.1]{Weiss} follows and hence $F^{(n)}$ is the (restriction to $\es{J}^\mbf{H}$) of the right Kan extension of $F$ along the inclusion $\es{J}^\mbf{H} \hookrightarrow \es{J}_n^\mbf{H}$. From this, we obtain evaluation maps
\[
\es{J}_n^\mbf{H}(V,V \oplus \mbf{H}) \wedge F^{(n)}(V) \longrightarrow F^{(n)}(V \oplus \mbf{H}),
\]
and by identifying $S^{4n}$ as a subspace of $\es{J}_n^\mbf{H}(V,V \oplus \mbf{H})$, the structure maps follow.
\end{proof}

\begin{cor}
If $F$ is $n$-polynomial, then the adjoint structure maps are equivalences, i.e., the maps
\[
F^{(n)}(V) \longrightarrow \Omega^{4n}F^{(n)}(V \oplus \mbf{H}),
\]
are equivalences for each $V$.
\end{cor}

\begin{rem}\label{rem: spectra with multiplicity}
Analysing this right Kan extension description for the $n$-th derivative of $F$, one sees that $F^{(n)}$ is naturally a module over the commutative monoid 
\[
n\bb{S} : V \longmapsto S^{nV} = S^{\mbf{H}^n \otimes V},
\]
in the category of $\mathsf{Sp}(n)$-equivariant functors from $\es{J}^\mbf{H}$ to $\mathsf{Sp}(n)$-equivariant pointed spaces. We won't make use of this perspective, but the arguments of Barnes and Oman~\cite{BarnesOman} or the second author~\cite{TaggartUnitary, TaggartReality} extend with little work to the quaternion situation. 
\end{rem}

Straightforward stable homotopy manipulations can, from a spectrum of multiplicity $4n$, produce a spectrum of multiplicity $1$, i.e., an ``honest'' spectrum, by ``inserting loops''. Details of this construction can be found in~\cite[\S2, \S3]{Weiss}. One way to view this is by viewing spectra with an action of $\mathsf{Sp}(n)$ as modules over $\bb{S} = 1\bb{S}$ and employ a ``extension-of-scalars'' adjunction. This latter is the approach used by Barnes and Oman~\cite[\S8]{BarnesOman} and the second author~\cite[\S6]{TaggartUnitary},\cite[\S6]{TaggartReality}.

\subsection{Classification of homogeneous functors} We now produce an equivalence of $\infty$-categories between the $\infty$-category of $n$-homogeneous functors and the $\infty$-category of spectra with an action of $\mathsf{Sp}(n)$. For the orthogonal calculus version of the following statement using the language of $\infty$-categories see for example,~\cite[Theorem 3.3.0.6]{Hendrian}.

\begin{thm}\label{thm: classification of homog}
    There is an equivalence of $\infty$-categories
    \[
    \homog{n}(\es{J}^\mbf{H}, \T_\ast) \simeq \s^{B\mathsf{Sp}(n)},
    \]
    between the $\infty$-category of $n$-homogeneous functors and the $\infty$-category of spectra with an action of the symplectic group $\mathsf{Sp}(n)$.
\end{thm}
\begin{proof}
Define a functor
\[
\s^{B\mathsf{Sp}(n)} \longrightarrow \Fun(\es{J}^\mbf{H}, \T_\ast), \qquad \Theta \longmapsto \Omega^\infty[(S^{nV} \wedge \Theta)_{hO(n)}].
\]
The content of~\cite[Theorem 7.3]{Weiss} apply to the quaternion case almost verbatim, allowing for us to identify the essential image of this functor with $\homog{n}(\es{J}^\mbf{H}, \T_\ast)$, and claim that the restricted functor
\[
\s^{B\mathsf{Sp}(n)} \longrightarrow \homog{n}(\es{J}^\mbf{H}, \T_\ast),\qquad \Theta \longmapsto \Omega^\infty[(S^{nV} \wedge \Theta)_{hO(n)}].
\]
is an equivalence of $\infty$-categories. 
\end{proof}

\begin{cor}
Denote by $\partial_n$ the composite of 
\[
\Fun(\es{J}^\mbf{H}, \T_\ast) \xrightarrow{\ D_n \ } \homog{n}(\es{J}^\mbf{H}, \T_\ast) \xrightarrow{ \ \simeq \ } \s^{B\mathsf{Sp}(n)},
\]
of $D_n$ with the inverse of the equivalence of Theorem~\ref{thm: classification of homog}. For each quaternion functor $F$, the spectrum $\partial_nF$ is equivalent to the $n$-th derivative of $F$ viewed as a spectrum with an action of $\mathsf{Sp}(n)=\Aut(\mbf{H}^n)$ in the sense of Remark~\ref{rem: spectra with multiplicity}.
\end{cor}


\begin{cor}
   For each $n \geq 1$ and each symplectic functor $F: \es{J}^\mbf{S} \to \T_\ast$ there is an equivalence
    \[
    D_nF(V) \simeq \Omega^\infty[(S^{nV} \wedge \partial_nF)_{h\mathsf{Sp}(n)}].
    \]
\end{cor}

\section{Comparing quaternion calculus with orthogonal and unitary calculus}
It is often useful to have methods for transporting problems through various versions of functor calculus. Observations of this kind were key to Arone's~\cite{AroneAkfree} work on finite type $n$ spectra, and to Behrens'~\cite{BehrensEHP} work on the relationship between the EHP sequence and Goodwillie calculus. In the series of papers~\cite{TaggartOCUC, TaggartUCReality, TaggartOCReality}, the second author proved that calculus with reality is a ``Galois extension'' of orthogonal calculus. In this section we add quaternion calculus to this list of comparison theorems by providing a comparison between quaternion calculus with unitary calculus, and hence also with orthogonal calculus.

There are adjunctions
\[\begin{tikzcd}
	{\es{J}^\mbf{O}} && {\es{J}^\mbf{U}} && {\es{J}^\mbf{H}},
	\arrow["{\mbf{C} \otimes_\R (-)}", shift left=2, from=1-1, to=1-3]
	\arrow["r", shift left=2, from=1-3, to=1-1]
	\arrow["{\mbf{H} \otimes_\mbf{C} (-)}", shift left=2, from=1-3, to=1-5]
	\arrow["h", shift left=2, from=1-5, to=1-3]
\end{tikzcd}\]
between the indexing categories for the various calculi. In~\cite{TaggartOCUC},  the second author provided an account of the relationship between orthogonal and unitary calculus showing that Weiss towers restrict along the realification functor $r: \es{J}^\mbf{U} \to \es{J}^\mbf{O}$. In this section, we prove the analogous statement for the functor
\[
h: \es{J}^\mbf{H} \longrightarrow \es{J}^\mbf{U}, \mbf{H}^r \longmapsto \mbf{C}^{2r}.
\]
The corresponding statements for the composite $rh$ will follow by combining the work of the second author~\cite{TaggartOCUC} with the work contained in this section. We leave the formulation of such statements to the interested reader. In fact, many of the arguments here are essentially replications of arguments used by the second author in~\cite{TaggartOCUC}, with slight modifications to improve some of the results.

\subsection{Polynomial functors}
The argument follows the now somewhat standard method for comparing calculi~\cite{BarnesEldred, TaggartOCUC}, by working along the Weiss tower. As such, we begin with the $n$-homogeneous functors.

\begin{lem}
    The functor $h^\ast$ preserves $n$-homogeneous functors, i.e., the functor $h^\ast$ induces a functor
    \[
    h^\ast : \homog{n}(\es{J}^\mbf{U}, \T_\ast) \longrightarrow \homog{n}(\es{J}^\mbf{H}, \T_\ast).
    \]
\end{lem}
\begin{proof}
    Argue with classifications, analogous to~\cite[Lemma 4.1]{TaggartOCUC}.
\end{proof}

We now show that the functor $h^\ast$ preserves polynomial functors. In~\cite{TaggartOCUC} a reduced assumption was necessary, here we provide an extra argument which removes this assumption. Our extra argument also applied in the context of~\cite{TaggartOCUC} implying that results there hold for all functors, not just the reduced ones.

\begin{lem}\label{lem: h preserves reduced polynomials}
    The functor $h^\ast$ preserves $n$-polynomial functors, i.e., the functor $h^\ast$ induces a functor
    \[
    h^\ast : \poly{n}(\es{J}^\mbf{U}, \T_\ast) \longrightarrow \poly{n}(\es{J}^\mbf{H}, \T_\ast),
    \]
    between the $\infty$-categories of $n$-polynomial functors in unitary calculus and $n$-polynomial functors in quaternion calculus.
\end{lem}
\begin{proof}
First assume that the $n$-polynomial functor $F$ is reduced, i.e., that $T_0F\simeq \ast$. Induction, analogous to~\cite[Theorem 4.3]{TaggartOCUC} yields that the canonical map
\[
h^\ast(T_n^\mbf{U}F) \longrightarrow T_n^\mbf{H}(h^\ast T_n^\mbf{U}F),
\]
is an equivalence whenever $F$ is reduced. Here the reduced assumption is essential for this argument as it uses the existence of the fiber sequence 
    \[
    T_nF \longrightarrow T_{n-1}F \longrightarrow R_nF, 
    \]
see for example,~\cite[Corollary 8.3, unravelled]{Weiss}. To remove the reduced assumption define the \emph{reduced part} $\mathsf{Red}(F)$ of a functor $F$ by the fiber sequence 
\[
\mathsf{Red}(F) \longrightarrow F \longrightarrow T_0F.
\]
Since $T_n$ preserves fiber sequences for all $n \geq 0$, applying $T_0$ to this fiber sequence shows $\mathsf{Red}(F)$ is a reduced functor. Since $h^\ast$ preserves fiber sequences an argument with the long exact sequence of homotopy groups implies that it suffices to show that the canonical map 
\[
h^\ast T_0^\mbf{U}F \longrightarrow T_n^\mbf{H}(h^\ast T_0^\mbf{U}F),
\]
is an equivalence. This follows since the functor $h^\ast T_0^\mbf{U}F$ is (homotopically) constant and hence $0$-polynomial, and in particular, by Lemma~\ref{prop: symplectic n-poly implies n+1-poly} also $n$-polynomial for every $n \geq 0$.
\end{proof}

The induction of Lemma~\ref{lem: h preserves reduced polynomials} yields the following corollary.

\begin{cor}
    Let $F$ be a functor in unitary calculus. The map
    \[
    h^\ast(T_n^\mbf{U}F) \longrightarrow T_n^\mbf{S}(h^\ast T_n^\mbf{U}F),
    \]
    is an equivalence. 
\end{cor}

\subsection{Weiss towers} We conclude this section by demonstrating the action of the functor $h^\ast$ on Weiss towers. For this we must assume our functor $F$ has convergent Weiss tower. We will use the formalism of ``weakly polynomial'' functors from~\cite[\S9]{TaggartUnitary}, which is built on the notion of ``agreement''.

\begin{definition}[{\cite[Definition 9.4]{TaggartUnitary}}]
Let $n$ be a non-negative integer. A natural transformation $p: F \longrightarrow G$ of unitary functors is \emph{an order $n$ unitary agreement} if there is some $\rho \in \mathbf{N}$ and $b \in \mbf{Z}$ such that $p_U: F(U) \longrightarrow G(U)$ is $(2(n+1)\dim(U)-b)$-connected for all $U \in \es{J}$, satisfying $\dim(U) \geq \rho$. We will say that \textit{$F$ agrees with $G$ to order $n$} if there is an order $n$ unitary agreement $p: F \longrightarrow G$ between them.
\end{definition}

Not that the above definition is phrased in terms of the \emph{complex} dimension of $U$, hence the additional factor of two. The connectivity bound could just as well be written as $(n+1)\dim_\mbf{R}(U)-b$.

\begin{definition}[{\cite[Definition 9.11]{TaggartUnitary}}]
    Let $n$ be a non-negative integer. A unitary functor $F$ is \emph{weakly $(\rho,n)$-polynomial} if the map $\eta : F(U) \longrightarrow T_nF(U)$ is an agreement of order $n$ whenever $\dim(U) \geq \rho$. A functor is \textit{weakly polynomial} if it is weakly $(\rho,n)$-polynomial for all $n\geq 0$. 
\end{definition}

\begin{thm}\label{thm: h on towers}
If $F$ be a weakly polynomial functor, then there is a levelwise equivalence
    \[
    h^\ast (T_n^\mbf{U} F) \simeq T_n^\mbf{H}(h^\ast F). 
    \]
\end{thm}
\begin{proof}
The proof is all but identical to \cite[Theorem 5.5]{TaggartOCUC}.
\end{proof}

The functor $h^\ast$ commutes with the inclusion $\es{J}_{/\mbf{H}^{n-1}} \hookrightarrow \es{J}_{/\mbf{H}^{n}}$ hence Theorem~\ref{thm: h on towers} has the following corollary.

\begin{cor}
    If $F$ is a weakly polynomial unitary functor, then there levelwise equivalence
    \[
    h^\ast \mathsf{Tow}^\mbf{U}(F) \simeq \mathsf{Tow}^\mbf{H}(h^\ast F). 
    \]
\end{cor}

\section{Classification of quaternion vector bundles}
As an application of the theory of quaternion calculus we classify certain quaternion vector bundles over finite-dimensional cell complexes in a range in terms of stable homotopy classes of maps. This is the quaternion version of~\cite{Hu}, which utilises unitary calculus to identity stably trivial vector bundles over a $d$-dimensional complex $X$ with $\{X, \Sigma \mbf{C} P^\infty_r\}$ in the metastable range, i.e., when a bundle has rank $r$ with $\frac{d}{4} \leq r \leq \frac{d-1}{2}$. Here $\{X, \Sigma \mbf{C} P^\infty_r\}$ denotes the set of stable maps between $X$ and the suspension of stunted complex projective space $\mbf{C}P^r=\mbf{C}P^\infty/\mbf{C}P^{r-1}$. Using a combination of the Adams and Atiyah-Hirzebruch spectral sequences Hu computes the stable spaces of maps in a range for $X = \mbf{C} P^\ell$, allowing for enumerations of stably trivial vector bundles. In this section we complete the quaternion version of Hu's program on classifying vector bundles by identifying stably trivial quaternion vector bundles over $X$ with $\{X, \Sigma^3 \mbf{H}P^\infty_r\}$ in a range, but we stop short of making any enumerations.

\subsection{The Weiss tower of $\mathsf{BSp}(-)$} Denote by $\mathsf{BSp}(-)$ the functor which sends a quaternion vector space $V$ to $\mathsf{BSp}(V)$, the classifying space of the compact symplectic group of $V$, i.e., the classifying space of the group of automorphisms of $V$. Using Definition~\ref{def: derivative}, we can readily compute the first derivative of $\mathsf{BSp}(-)$ in compact symplectic calculus, and hence the first layer of the associated Weiss tower.

\begin{lem}
The first layer of the Weiss tower of $\mathsf{BSp}(-)$ evaluated at $\mbf{H}^r$ is equivalent to $\Omega^\infty\Sigma^\infty\Sigma^3 \mbf{H}P^\infty_r$. 
\end{lem}
\begin{proof}
By the classification of homogeneous functors in Theorem~\ref{thm: classification of homog}, the first layer of the Weiss tower of $\mathsf{BSp}(-)$ evaluated at $\mbf{H}^r$ is equivalent to
\[
\Omega^\infty (S^{4r} \wedge \partial_1 (\mathsf{BSp}(-)))_{h\mathsf{Sp}(1)}.
\]
The fiber sequence 
\[
\Sigma^3S^{\mbf{H}^r}=S^{4r+3} \longrightarrow \mathsf{BSp}(r) \longrightarrow \mathsf{BSp}(r+1),
\]
identifies $\partial_1(\mathsf{BSp}(-))$ with $\Sigma^3\bb{S}$ with trivial $\mathsf{Sp}(1)$-action, and the proof follows from observing
\[
\Omega^\infty (S^{4r} \wedge \partial_1 (\mathsf{BSp}(-)))_{h\mathsf{Sp}(1)} \simeq \Omega^\infty\Sigma^{\infty+3} (S^{4r})_{h\mathsf{Sp}(1)} \simeq \Omega^\infty\Sigma^\infty\Sigma^3 \mbf{H}P^\infty_r), 
\]
where $\mbf{H}P^\infty_r$ is the stunted quaternion projective space defined as the cofiber of the canonical inclusion $\mbf{H}P^{r-1} \hookrightarrow \mbf{H}P^\infty$.
\end{proof}

We now show that the Weiss tower of $\mathsf{BSp}(-)$ converges, using the quaternion version of weakly polynomial functors introduced by the second author~\cite[\S9]{TaggartUnitary}.

\begin{lem}
The Weiss tower of $\mathsf{BSp}(-)$ converges at $\mbf{H}^r$ to $\mathsf{BSp}(r)$ for $r \geq 1$.
\end{lem}
\begin{proof}
    By the quaternion version of~\cite[Corollary 9.15]{TaggartUnitary}, it suffices to show that the functor 
    \[
    V \longmapsto \Sigma^3 S^V,
    \]
    is weakly polynomial in the sense of~\cite[Definition 9.11]{TaggartUnitary}. This follows from the quaternion version of~\cite[Example 9.13]{TaggartUnitary}.
\end{proof}

\begin{cor}
    The $n$-th layer of the Weiss tower of $\mathsf{BSp}(-)$ evaluated at $\mbf{H}^r$ is at least $(4nr-1)$-connected, i.e.,
    \[
    \mathsf{Conn}(D_n\mathsf{BSp}(r)) \geq 4nr-1.
    \]
\end{cor}
\begin{proof}
    For any weakly polynomial functor $F$ with excess $c$, a standard argument with the Weiss tower yields that $D_nF(V)$ is at least $(n(\dim_\R(V))-c)$-connected, for some constant $c$. The proof follows by identifying $F = \mathsf{BSp}(-)$ and $V = \mbf{H}^r$, the the fact that the identity functor $\id: \T_\ast \to \T_\ast$ satisfies $E_n(n)$ in the sense of~\cite[Definition 4.1]{GoodCalcII} for all $n$, see for example,~\cite[Example 4.3]{GoodCalcII}, and hence $O_n(n)$ in the sense of~\cite[Definition 1.2]{GoodCalcIII}, see also~\cite[Proposition 1.5]{GoodCalcIII}. 
\end{proof}

\subsection{The Weiss tower of $\mathsf{BSp}^X(-)$}
In this section $X$ will be a finite-dimensional cell complex, and we will study the Weiss tower of the functor
\[
\mathsf{BSp}^X(-): V \longmapsto \Map_\ast(X, \mathsf{BSp}(V)).
\]
By applying $\pi_0$ to this functor we get
\[
\pi_0(\mathsf{BSp}^X(\mbf{H}^r)) = \pi_0\Map_\ast(X, \mathsf{BSp}(r))= [X, \mathsf{BSp}(r)],
\]
which by the classification of quaternionic vector bundles over $X$, is equivalent to the set of rank $r$ bundles over $X$. The set of stably trivial vector bundles is (under niceness conditions on $X$) the kernel of the canonical map
\[
[X, \mathsf{BSp}(r)] \longrightarrow [X, \mathsf{BSp}],
\]
and the rest of this section is dedicated to identifying this with $\pi_0\Map(X, D_1\mathsf{BSp}(r))$, yielding the result.

\begin{rem}
Much of what we say will be true when $\mathsf{BSp}(-)$ is replaced by an arbitrary (analytic) functor.
\end{rem}

\begin{lem}
Let $X$ be a finite-dimensional cell complex. The $n$-polynomial approximation of the functor 
\[
\mathsf{BSp}^X(-): V \longmapsto \Map_\ast(X, \mathsf{BSp}(V)).
\]
is equivalent to the functor 
\[
(P_n\mathsf{BSp})^X(-): V \longmapsto \Map_\ast(X, P_n\mathsf{BSp}(V)).
\]
In particular, the functor
\[
V \longmapsto \Map_\ast(X, \mathsf{BSp}(V)).
\]
is weakly polynomial.
\end{lem}
\begin{proof}
    The first statement follows from the properties of maps from a finite cell complex to filtered homotopy colimits and finite homotopy limits. The second statement follows immediately from the properties of mapping from a finite cell complex to a finite homotopy limit. 
\end{proof}

\subsection{The Weiss tower of $\mathsf{BSp}^X(-)$ in a range} Hu~\cite{Hu} provided a classification of certain rank $r$ complex bundles over a $d$-dimensional complex whenever $\frac{d}{4} \leq r$. In this section we provide a similar classification result for certain rank $r$ quaternion bundles over a $d$-dimensional complex whenever $\frac{d+1}{8}\leq r$.  Denote by $\mathsf{KSp}$ quaternionic topological $K$-theory and by $\widetilde{\mathsf{KSp}}$ the reduced version of quaternionic topological $K$-theory. 

\begin{lem}
Let $X$ be a $d$-dimensional cell complex for which $\widetilde{\mathsf{KSp}}^{-1}(X) =0$. On connected components, the Weiss tower of $\mathsf{BSp}^X(-)$ stabilises at the first stage, when evaluated at $\mbf{H}^r$ with $\frac{d+1}{8}\leq r < \frac{d-2}{4}$, that is, for $d$ and $r$ as stated, the Weiss tower of $\mathsf{BSp}^X(-)$ on connected components is the exact sequence
\[
0 \longrightarrow [X, Q\Sigma^3\mbf{H}P^\infty_r] \longrightarrow [X, \mathsf{BSp}(r)] \longrightarrow [X, \mathsf{BSp}],
\]
where $\mathsf{BSp} = \bigcup_{r \geq 0} \mathsf{BSp}(r)$.
\end{lem}
\begin{proof}
Since $X$ is $d$-dimensional and $d$ and $r$ are bounded as above, the $n$-th layer of the Weiss tower of $\mathsf{BSp}^X(\mbf{H}^r)$ is connected for all $n \geq 2$. It follows that on $\pi_0$, 
\[
\pi_0(\mathsf{BSp}^X(\mbf{H}^r)) \cong \pi_0(\holim_n T_n\mathsf{BSp}^X(\mbf{H}^r)) \cong \pi_0(T_1\mathsf{BSp}^X(\mbf{H}^r)),
\]
and hence the fiber sequence 
\[
D_1\mathsf{BSp}^X(\mbf{H}^r) \longrightarrow T_1\mathsf{BSp}^X(\mbf{H}^r) \longrightarrow T_0\mathsf{BSp}^X(\mbf{H}^r),
\]
on $\pi_0$ is the exact sequence
\[
0 \longrightarrow [X, Q\Sigma^3\mbf{H}P^\infty_r] \longrightarrow [X, \mathsf{BSp}(r)] \longrightarrow [X, \mathsf{BSp}].\qedhere
\]
\end{proof}

\begin{cor}
The set of stably trivial quaternion vector bundles of rank $r$ over a $d$-dimensional cell complex $X$ with $\frac{d+1}{8}\leq r < \frac{d-2}{4}$ and $\widetilde{\mathsf{KSp}}^{-1}(X)=0$ is in bijection with the set of stable maps from $X$ to $\Sigma^3\mbf{H}P^\infty_r$, the three-fold suspension of the stunted projective space. 
\end{cor}


\bibliographystyle{alpha}
\bibliography{references}
\end{document}